\documentclass{article}

\usepackage[english,american]{babel}

\usepackage{url}
\usepackage{epsfig} 
\usepackage{times} 
\usepackage{amsmath}
\usepackage{amsfonts}
\usepackage{amssymb}  

\usepackage{xcolor}
\usepackage{transparent}
\usepackage{graphicx}
\graphicspath{{media/}}
\usepackage{tikz}
\usetikzlibrary{shapes,arrows}
\usetikzlibrary{positioning}

\usepackage{epstopdf}
\usepackage{units}
\usepackage{pgfplots}

\usepackage{pstool}

\graphicspath{{media/}}
\newcommand{\executeiffilenewer}[3]{%
 \ifnum\pdfstrcmp{\pdffilemoddate{#1}}%
 {\pdffilemoddate{#2}}>0%
 {\immediate\write18{#3}}\fi%
}

\usepackage{arydshln}
\usepackage{perpage}
\MakePerPage{footnote}

\newtheorem{theorem}{Theorem}[section]
\newtheorem{lemma}[theorem]{Lemma}
\newtheorem{proposition}[theorem]{Proposition}

\newenvironment{proof}[1][Proof]{\begin{trivlist}
\item[\hskip \labelsep {\bfseries #1}]}{\end{trivlist}}

\newcommand{\qed}{\nobreak \ifvmode \relax \else
      \ifdim\lastskip<1.5em \hskip-\lastskip
      \hskip1.5em plus0em minus0.5em \fi \nobreak
      \vrule height0.75em width0.5em depth0.25em\fi}

\pgfplotsset{every tick label/.append style={font=\footnotesize}}
\pgfdeclarelayer{background}
\pgfsetlayers{background,main}

\newcommand{\TT}{\ensuremath{\mathsf{\tiny{T}}}}
\newcommand{\T}{^{\TT}}

\newcommand{\diff}[1][]{\mathrm{d}#1}

\usepackage{microtype}
\usepackage{subfigure}
\usepackage{booktabs} 

\usepackage{hyperref}

\usepackage[accepted]{arxivVersion} 

\icmltitlerunning{Continuous-time Lower Bounds}

\begin{document}

\twocolumn[
\icmltitle{Continuous-time Lower Bounds for Gradient-based Algorithms}




\begin{icmlauthorlist}
\icmlauthor{Michael Muehlebach}{berk}
\icmlauthor{Michael I. Jordan}{berk}
\end{icmlauthorlist}

\icmlaffiliation{berk}{Division of Electrical Engineering and Computer Science, and Department of Statistics, University of California, Berkeley, Berkeley, USA}

\icmlcorrespondingauthor{Michael Muehlebach}{michaelm@berkeley.edu}

\icmlkeywords{mathematical optimization, continuous-time lower bounds, non-convex optimization}

\vskip 0.3in
]



\printAffiliationsAndNotice{}  

\hyphenation{dy-na-mi-cal}
\hyphenation{cir-cum-stances}

\begin{abstract}
This article derives lower bounds on the convergence rate of continuous-time gradient-based optimization algorithms. The algorithms are subjected to a time-normalization constraint that avoids a reparametrization of time in order to make the discussion of continuous-time convergence rates meaningful. We reduce the multi-dimensional problem to a single dimension, recover well-known lower bounds from the discrete-time setting, and provide insight into why these lower bounds occur. We present algorithms that achieve the proposed lower bounds, even when the function class under consideration includes certain nonconvex functions.
\end{abstract}

\section{Introduction}
Many problems in machine learning and statistics can be formulated as optimization problems. First-order optimization algorithms, such as gradient descent, are commonly used due to their simplicity and due to the fact that their complexity scales mildly in the number of decision variables. These algorithms are known to have natural limits on their convergence rate. We will examine these complexity lower bounds from a dynamical systems perspective.  Unusually for the literature on lower bounds, we will work in continuous time. Our continuous-time perspective not only leads to insights into why complexity bounds arise, but also provides guidance for algorithm design.

A commonly used technique for deriving lower bounds is to construct a function that is difficult to optimize~\citep[see, e.g.,][p.~59]{NesterovBook}. Such a function is typically obtained by ensuring that, at the $j$th iteration, all gradients that an algorithm has evaluated so far belong to a $(j+1)$-dimensional subspace which is far away from the optimum. Dimension-independent lower bounds then result from an unbounded increase in the problem dimension. This establishes, for example, that for the class of smooth and strongly convex functions at least $\sqrt{\kappa}/2~\text{ln}(c/\epsilon)$ iterations are needed to achieve an $\epsilon$-distance to the optimizer (for large $\kappa$), where $c$ is a constant and $\kappa$ refers to the condition number \citep[][p.~68]{NesterovBook}. The lower bound is achieved by accelerated gradient methods; for example, \citet{TripleMomentum} provide an algorithm that attains the lower bound up to a factor of two. In the nonconvex setting, deriving tight lower bounds for smooth functions is an active area of research, where important recent contributions have been made by \citet{Duchi}, for example.

We are motivated by a line of recent work that views optimization algorithms as continuous-time dynamical systems \citep[see, e.g.,][]{SuAcc, KricheneAcc, WibisonoVariational, Diakonakolis,MuehlebachJordan}.  This work has provided significant insight into convergence rates of discrete-time algorithms via translating \emph{upper bounds} from continuous time to discrete time. We ask the question whether it is also possible to obtain insight into \emph{lower bounds} on gradient-based algorithms via a continuous-time analysis. 

Instead of constructing a function that is difficult to optimize, we exploit invariance properties of the function class under consideration, which, combined with a local analysis about a (local) minimum, greatly simplifies the dynamics that need to be considered. This reduces the problem of determining the worst-case convergence rate over the given class of functions to the analysis of a parameter-dependent characteristic polynomial. We show that under certain circumstances the classical dimension-independent discrete-time lower bound for smooth and strongly convex functions can be recovered. We also derive continuous-time algorithms that achieve faster convergence rates. These algorithms include very fast dynamics and it is a matter of future research to investigate whether it is possible to derive practical discretizations of these dynamics. 

A related---but discrete-time---perspective has been presented by \citet{Arjevani}, where optimization algorithms are modeled by $k$th-order linear dynamical systems, and the complexity is shown to be lower bounded by $\Omega(\kappa^{1/k}~ \text{ln}(1/\epsilon))$.\footnote{Here, $k$th order refers to the fact that the $j$th iterate depends only on the past $k$ iterates. This should not be confused with accessing higher derivatives of the objective function.} In contrast, we model algorithms as continuous-time nonlinear dynamical systems and show how the lower bound $\Omega(\sqrt{\kappa}~\text{ln}(1/\epsilon))$ for the class of strongly convex quadratic functions can be recovered.

\subsection{Notation and outline}
We focus on optimizing real-valued functions $f:\mathbb{R}^n \rightarrow \mathbb{R}$, where $n>0$ is an integer. Without loss of generality, we assume that the functions $f$ have a local minimum at $x=0$ with value $f(0)=0$. The functions are assumed to have Lipschitz-continuous gradients. Our aim is to find a lower bound on the convergence rate that any continuous-time gradient-based algorithm can possibly achieve on a given class of functions. This class of functions will be denoted by $C_{\mu,L}$ and is required to satisfy the following assumptions:
\vspace{-5pt}
\begin{itemize}
\item[(C1)] Each $f\in C_{\mu,L}$ is twice continuously differentiable in a neighborhood of the origin.
\item[(C2)] For every $f\in C_{\mu,L}$, it holds that 
\begin{equation*}
\text{spec}(\Delta f(0)) \subset [\mu,L],
\end{equation*}
where $0<\mu\leq L$ are fixed constants, $\text{spec}$ denotes the spectrum, and $\Delta f(0)$ refers to the Hessian of the function $f$ evaluated at the origin. Conversely, for any $\lambda_f\in [\mu,L]$, there exists a function $f\in C_{\mu,L}$ such that 
\begin{equation*}
\lambda_f\in \text{spec}(\Delta f(0)).
\end{equation*}
\item[(C3)] The class $C_{\mu,L}$ is invariant under orthogonal transformations. In other words, $f\in C_{\mu,L}$ implies that $f\circ T\in C_{\mu,L}$ for all $T\in O(n)$, where $O(n)$ denotes the set of orthogonal matrices of size $n\times n$.
\end{itemize}
\vspace{-5pt}
Assumption (C1) imposes local smoothness and Assumption (C2) encodes prior information about the local curvature. It excludes degeneracies, which arise either due to a non-isolated minimum, or when the curvature about the minimum is arbitrarily small. Assumption (C3) implies that the function class $C_{\mu,L}$ is invariant under permutations and rotations. As we shall see in the sequel, this has important implications for algorithm design. 

Given these assumptions, the lower bounds that we derive  apply to smooth nonconvex functions with isolated non-degenerate critical points, convex quadratic functions, and smooth and strongly convex functions. The assumptions emphasize the importance of the local shape of the objective function about a local minimum. From a dynamical systems perspective imposing limits on the local instead of the global structure seems more natural. Even though our analysis includes results about certain nonconvex functions, we will not consider the impact of saddle points, for example, which greatly limits the convergence rate~\cite{Chi,Duchi}.

The complexity of an algorithm can be characterized by the number of iterations required to achieve an $\epsilon$-distance to the optimizer. In the following, it will be more convenient to characterize the convergence rate. We say that an algorithm converges with rate $\rho>0$, if the distance to the local optimum decays at least with $\exp(-\rho t)$ for a certain set of initial conditions, where $t$ refers to time. Both notions are equivalent. However, an upper bound on the convergence rate leads to a lower bound on the complexity, and vice versa. For example, if the convergence rate is upper bounded by $\mathcal{O}(1/\kappa)$, the complexity is lower bounded by $\Omega(\kappa~\text{ln}(1/\epsilon))$.

The article is structured as follows: Sec.~\ref{Sec:mod} introduces the class of algorithms that are studied. The resulting lower bounds are presented in Sec.~\ref{Sec:LBsec} and simulation results are provided in Sec.~\ref{Sec:simRes}. The article concludes with a brief discussion in Sec.~\ref{Sec:Conclusion}.

\section{Continuous-time Gradient-based Optimization Algorithms}\label{Sec:mod}
We model a gradient-based optimization algorithm as a dynamical system of the form
\begin{multline}
x^{(k)}(t)=g(x^{(k-1)}(t), \dots, \dot{x}(t), x(t), \\
	\nabla f[h(x(t),\dot{x}(t),\dots, x^{(k-1)}(t))~]~), \label{eq:ode}
\end{multline}
where the functions $g: \mathbb{R}^{n\times k}\times \mathbb{R}^n \rightarrow \mathbb{R}^n$ and $h:\mathbb{R}^{n\times k} \rightarrow \mathbb{R}^n$ are independent of $f$, where $k>0$ is an integer, and where $x^{(p)}(t)$ denotes the $p$th derivative with respect to time. We say that $(g,h)$ is a continuous-time gradient-based optimization algorithm, $(g,h)\in \mathcal{G}$, if \eqref{eq:ode} has the following properties (for all $f\in C_{\mu,L}$):
\vspace{-5pt}
\begin{itemize}
\item[(G1)] $g$ and $h$ are continuously differentiable in all arguments,
\item[(G2)] critical points of $f$ correspond to equilibria of \eqref{eq:ode},
\item[(G3)] local minima of $f$ correspond to asymptotically stable equilibria of \eqref{eq:ode} (in the sense of Lyapunov).
\end{itemize}
\vspace{-5pt}
Assumption (G1) implies that \eqref{eq:ode} is a well-posed differential equation. Assumption (G2) and (G3) ensures that the algorithm locally converges to local minima. Assumption (G1)-(G3) are therefore minimal requirements for ensuring that \eqref{eq:ode} minimizes $f$. A graphical representation of the system \eqref{eq:ode} is shown in Fig.~\ref{Fig:Feedback}.

\emph{Remark:} We model a gradient-based optimization algorithm as an autonomous dynamical system. On a fundamental level, introducing non-autonomous dynamics would lead to a time-varying vector field, which contrasts with the fact that the objective function $f$ is fixed. In physics and engineering, non-autonomous dynamical systems typically arise in situations where only a sub-component of a system is studied. In that case, the interactions of the sub-component with the rest might lead to non-autonomous dynamics, even though the system as a whole is autonomous. Thus, non-autonomous dynamics, as introduced in \citet{SuAcc}, for example, might be useful to approximate \eqref{eq:ode} with a reduced-order model, whereby higher-order dynamics are captured by time-varying terms. In the following, we will characterize fundamental limits on the convergence rate for any integer $k>0$; therefore there is no need to render $g$ and $h$ time varying.

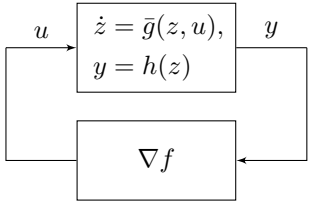
\begin{figure}
\center
\tikzstyle{block} = [draw, rectangle, 
    minimum height=3em, minimum width=6em]
\tikzstyle{sum} = [draw, fill=blue!20, circle, node distance=1cm]
\tikzstyle{input} = [coordinate]
\tikzstyle{output} = [coordinate]
\tikzstyle{pinstyle} = [pin edge={to-,thin,black}]

\begin{tikzpicture}[auto, node distance=2cm,>=latex']
    \node [block] (system) {$\begin{aligned}
            								\dot{z}&=\bar{g}(z,u),\\
            								 y&=h(z)
            								 \end{aligned}$};
            								 
    \node [block, below of=system,node distance=1.5cm] (grad) {$\nabla f$};

    \node (rn) [coordinate,right of=system] {};
 	\node (ln) [coordinate,left of=system] {};
 	\draw [-] (system) -- node[name=y]{$y$} (rn);   
 	\draw [->] (rn) |- (grad); 
	\draw [-] (grad) -| (ln);
	\draw [->] (ln) -- node[name=u]{$u$} (system);
\end{tikzpicture}
\caption{Graphical representation of the dynamics \eqref{eq:ode}, where $z$ denotes the internal state, and $\bar{g}$ is related to $g$. The feedback connection with the gradient $\nabla f$ can be viewed as an oracle query. The assumptions on the structure of the dynamical system (as given by \eqref{eq:ode}) are without loss of generality.}
\label{Fig:Feedback}
\end{figure}

\subsection{Rescaling time}
In contrast to the discrete-time setting, the time $t$ has no absolute meaning in continuous time. The trajectory $\tilde{x}(t):=x(\alpha(t))$, where $\alpha: \mathbb{R}_{\geq 0} \rightarrow \mathbb{R}_{\geq 0}$ is any diffeomorphism satisfies a differential equation similar to \eqref{eq:ode}. 

For $k=2$, for example, the trajectory $\tilde{x}(t)$ evolves according to
\begin{equation}\label{eq:transform}
\ddot{\tilde{x}}=g(\dot{\tilde{x}}/\dot{\alpha}, \tilde{x}, \nabla f[h(\tilde{x},\dot{\tilde{x}}/\dot{\alpha})]) \dot{\alpha}^2 + \dot{\tilde{x}} \frac{\ddot{\alpha}}{\dot{\alpha}},
\end{equation}
where the dependence on time has been omitted.
In order for a discussion of convergence rates to be meaningful, such a rescaling needs to be avoided. This is done by adding the following requirement: $(g,h)\in \mathcal{G}'$, 
\begin{equation}\label{eq:timenorm}
\mathcal{G}'\!\!:=\!\!\Big\{(g,h)\in \mathcal{G} ~|~\text{det}\left(\!\left.\frac{\partial g(w,v)}{\partial v}\right|_{0,0}\! \left.\frac{\partial h}{\partial x}\right|_{0}\right) \! =\! (-1/L)^n \Big\},
\end{equation}
where the variables $w\in \mathbb{R}^{n\times k}$ and $v\in \mathbb{R}^n$ are placeholders, and $\text{det}$ denotes the determinant. In the example of $k=2$, the right-hand-side of \eqref{eq:transform} satisfies the time-normalization constraint \eqref{eq:timenorm} if and only if
\begin{equation*}
\left(\frac{-1}{L}\right)^n\!\!=\!\text{det}\left(\!\left.\frac{\partial g(w,v)}{\partial v}\right|_{0,0} \!\!\left.\frac{\partial h}{\partial x}\right|_{0}\!\right) \dot{\alpha}(t)^{2n}\!
=\!\!\left(\frac{-\dot{\alpha}(t)^{2}}{L}\right)^n\!\!,
\end{equation*}
where $(g,h)\in \mathcal{G}'$ has been used for the last equality. This implies $\dot{\alpha}(t)^2=1$ for all $t\in \mathbb{R}_{\geq 0}$, or equivalently, $\alpha(t)=t+\text{const}$. The same argument applies for $k=1$ or $k>2$ and implies that \eqref{eq:timenorm} fixes the time scale.

Choosing any other normalization in \eqref{eq:timenorm}, such as the trace, the induced two-norm, a specific entry, or a constant different than $1/L$, fixes the time scale in a different way; however, the results derived in the remainder of the paper still apply. The normalization according to \eqref{eq:timenorm} is convenient, since, as a result, a convergence rate of unity is achieved for the class $C_{L,L}$, which consists of functions that behave locally like $L |x|^2/2$. The sign is imposed by the asymptotic stability requirement (G3). Additional context on the normalization \eqref{eq:timenorm} is provided in App.~\ref{App:TimeNorm}.

\subsection{First-order approximation}
For deriving the lower bounds on the convergence rate it will be enough to consider initial conditions that are close to the origin. We therefore apply Taylor's theorem to the dynamics \eqref{eq:ode},
\begin{multline}
x^{(k)}=-\sum_{j=1}^{k-1} G_j x^{(j)}-G_0 \Delta f(0) \left( x+ \sum_{j=1}^{k-1} H_j x^{(j)} \right)\\
+r(x,\dot{x},\dots,x^{(k-1)}),\label{eq:taylor}
\end{multline}
where the dependence on time is omitted, and the matrices $G_0,\dots,G_{k-1} \in \mathbb{R}^{n\times n}$ and $H_1,\dots H_{k-1}\in \mathbb{R}^{n\times n}$ represent the different partial derivatives of $g$ and $h$ evaluated at the origin. The remainder term is denoted by the function $r: \mathbb{R}^{n\times k} \rightarrow \mathbb{R}^n$, and captures the second-order terms. In deriving \eqref{eq:taylor}, we exploited the fact that the partial derivative of $g$ with respect to $x$ vanishes, when evaluated at the origin, due to Assumption (G2). Furthermore, we set the partial derivative of $h$ with respect to $x$, evaluated at the origin, to the identity. This amounts to a normalization of the state $x$, which, according to Assumption (G2) can always be done and does not affect the convergence rate.

By introducing the state variable $z:=(x,\dot{x},\dots,x^{(k-1)})$, the dynamics \eqref{eq:taylor} can be rewritten as
\begin{equation}
\dot{z}(t)=A z(t)+ \tilde{r}(z(t)),\label{eq:lineq}
\end{equation}
where $A\in \mathbb{R}^{kn \times kn}$ is obtained by appropriately stacking the matrices $G_0,\dots, G_{k-1}$, $H_1,\dots,H_{k-1}$ and $\tilde{r}$ captures the remainder term. The following lemma relates the convergence rate of the nonlinear dynamics to the convergence rate of the linear dynamics.
\begin{lemma}\label{Lem:lin}
Assume that there exists a neighborhood $N$ of the origin such that any solution $z(t)$ of \eqref{eq:lineq} with $z(0)\in N$ satisfies 
\begin{equation*}
|z(t)|\leq c_1 |z(0)| \exp(-a_\text{r} t),\quad \forall t\in [0,\infty),
\end{equation*}
where $c_1\geq 1$ and $a_\text{r}>0$ are constants. Then, there exists a constant $c_2\geq 1$, such that any solution of the corresponding linear equation $\Delta \dot{z}(t)=A \Delta z(t)$ satisfies the estimate
\begin{equation*}
|\Delta z(t)| \leq c_2 |\Delta z(0)| \exp(-a_\text{r} t).
\end{equation*} 
\end{lemma}

The lemma is a standard result from the theory of ordinary differential equations. We included a proof in App.~\ref{App:Lemma}.
A lower bound on the convergence rate of the linearized dynamics provides us therefore with a lower bound on the convergence of the nonlinear dynamics, as this would otherwise contradict the statement of Lemma~\ref{Lem:lin}. In order to find lower bounds on the convergence rate, we will therefore replace \eqref{eq:ode} with its first-order approximation \eqref{eq:taylor}, where we neglect the remainder term $r$.

\subsection{Invariance under orthogonal transformations}\label{Sec:invar}
Assumption (C3) requires that the class $C_{\mu,L}$ is invariant under orthogonal transformations. We will show next that, without loss of generality, the linearized dynamics \eqref{eq:taylor} (when $r$ is neglected) can be assumed to be invariant under orthogonal transformations, which will simplify our derivations.


\begin{proposition}\label{Prop:Inv}
Assume \eqref{eq:taylor} converges with rate $\rho$ (or faster) for all $f\in C_{\mu,L}$. Then, the rate $\rho$ can be achieved for $G_0=g_0 I,\dots, G_{k-1}=g_{k-1} I$, $H_1=h_1 I, \dots, H_{k-1}=h_{k-1} I$, where $g_0,\dots,g_{k-1}$ and $h_1,\dots,h_{k-1}$ are scalars, and $I\in \mathbb{R}^{n\times n}$ is the identity.
\end{proposition}
\begin{proof}
The convergence rate of \eqref{eq:taylor} (with $r=0$) is determined by the roots of the characteristic polynomial
\begin{equation*}
\text{det}\left(s^kI+\sum_{j=1}^{k-1} G_j s^j + G_0 \Delta f(0) \left( I+ \sum_{j=1}^{k-1} H_j s^j\right)\right).
\end{equation*}
We set $\Delta f(0)=\lambda I$ and analyze the characteristic polynomial for different $\lambda \in [\mu,L]$. The polynomial can be factorized into $n$ factors, each having the form
\begin{equation}
s^k + \tilde{g}_{k-1}(\lambda) s^{k-1} + \dots + \tilde{g}_1(\lambda) s + \lambda \tilde{g}_0, \label{eq:fundPoly}
\end{equation}
where the coefficients $\tilde{g}_j(\lambda)$ continuously depend on $\lambda$ and $\tilde{g}_0$ is given by an eigenvalue of $G_0$. The latter follows from evaluating the above determinant at $s=0$. By assumption, each of these factors must have roots with real parts smaller than $-\rho$ for all $\lambda\in [\mu,L]$. This implies, by Kharitonov's theorem \cite{Kharitonov}, that the four Kharitonov polynomials, which are given by different combinations of the maxima and minima of $\tilde{g}_j(\lambda)$ over $\lambda\in [\mu,L]$, have real parts smaller than $-\rho$. We pick any of the $n$ factors, \eqref{eq:fundPoly}, and choose $g_0,\dots,g_{k-1}$ and $h_1,\dots,h_{k-1}$ such that $g_0=\tilde{g}_0$,
\begin{equation*}
g_j+\mu h_j=\hspace{-5pt}\min_{\lambda \in [\mu,L]} \tilde{g}_j(\lambda),~~ g_j+L h_j=\hspace{-5pt}\max_{\lambda \in [\mu,L]} \tilde{g}_j(\lambda),
\end{equation*}
$j=1,\dots,k-1$.
In that case, for any $f\in C_{\mu,L}$, the characteristic polynomial factorizes into $n$ equal factors, each having the form
\begin{equation}
s^k+(g_{k-1} + h_{k-1} \bar{\lambda}) s^{k-1} +\dots + \bar{\lambda} g_0, \label{eq:fundPoly2}
\end{equation}
where $\bar{\lambda}\in [\mu,L]$ is an eigenvalue of $\Delta f(0)$. By construction, the Kharitonov polynomials of \eqref{eq:fundPoly} and \eqref{eq:fundPoly2} agree, which guarantees a convergence rate of at least $\rho$. \qed
\end{proof}

The fact that a square matrix is invariant under orthogonal transformations if and only if it is a scaled identity matrix implies that the dynamics \eqref{eq:taylor} are invariant under orthogonal transformations of $f$ if and only if the matrices $G_{0},\dots, H_{k-1}$ are scaled identity matrices. A game-theoretic interpretation of Prop.~\ref{Prop:Inv} is included in App.~\ref{App:InterInv}.

We therefore conclude that for any $(g,h)\in \mathcal{G}'$, a lower bound on the convergence rate is obtained by the first-order approximation of \eqref{eq:ode}, which, due to the time-normalization constraint and the invariance under orthogonal transformations, takes the form
\begin{equation}
x^{(k)} = -\sum_{j=1}^{k-1} g_j x^{(j)} -\frac{1}{L} \Delta f(0) \Big( x + \sum_{j=1}^{k-1} h_j x^{(j)} \Big), \label{eq:first-orderNorm}
\end{equation}
where $g_1,\dots,g_{k-1}$, $h_1,\dots,h_{k-1}\in \mathbb{R}$ are scalars and the dependence on time has been omitted. The time-normalization constraint \eqref{eq:timenorm} implies $g_0=1/L$.

\section{Lower Bounds on the Convergence Rate}\label{Sec:LBsec}
\subsection{Gradient flow ($k=1$)}\label{Sec:GradientFlow}
This section analyses the case $k=1$, resulting in gradient-flow algorithms. According to \eqref{eq:first-orderNorm}, we obtain the first-order approximation
\begin{equation}
\dot{x}(t)= -\frac{1}{L}  \Delta f(0) x(t).\label{eq:odegrad}
\end{equation}
By choosing an appropriate coordinate system that diagonalizes $\Delta f(0)$, we conclude that each component $x_j$ of $x$ converges according to 
\begin{equation}
x_j(t) = x_j(0) \exp(-\lambda_f t),
\end{equation}
where $\lambda_f$ is an eigenvalue of $\Delta f(0)/L$. Due to the fact that $f\in C_{\mu,L}$, which implies $1/\kappa \leq \lambda_f \leq 1$, according to (C2), $x(t)$ converges in the worst-case with $c \exp(-t/\kappa)$, where $\kappa:=L/\mu$ and $c$ is constant. Hence, by virtue of Lemma~\ref{Lem:lin} we conclude:

\emph{The convergence rate of any continuous-time optimization algorithm $(g,h)\in \mathcal{G}'$ (with $k=1$) is upper bounded by $1/\kappa$ on functions $C_{\mu,L}$. Thus, the time required to achieve an $\epsilon$-distance to the optimizer is lower bounded by $\kappa~\text{ln}(c/\epsilon)$ on functions $C_{\mu,L}$, where $c>0$ is constant.}

\subsection{Accelerated gradient flow ($k=2$)}\label{Sec:AccGradientFlow}
This section discusses the case $k=2$. According to \eqref{eq:first-orderNorm}, the following first-order approximation is obtained:
\begin{equation}
\ddot{x}(t)=- \left( g_1 + \frac{h_1}{L} \Delta f(0)\right) \dot{x}(t) - \frac{1}{L} \Delta f(0) x(t).
\end{equation}
Choosing an appropriate coordinate system that diagonalizes $\Delta f(0)$, results in the scalar second-order differential equation for each component $x_j$ of $x$,
\begin{equation}
\ddot{x}_j(t)=-(g_1+h_1 \lambda_f)\dot{x}_j(t) - \lambda_f x_j(t),
\end{equation}
where $\lambda_f$ is an eigenvalue of $\Delta f(0)/L$. The convergence rate of a single component of $x$ is therefore dictated by the real part of the roots of the following polynomial:
\begin{equation}
s^2+(g_1+h_1 \lambda_f) s + \lambda_f =0.
\end{equation}
The roots are given by
\begin{equation}
-d_{\lambda_f}(g_1,h_1) \pm \sqrt{d_{\lambda_f}(g_1,h_1)^2 - \lambda_f}, \label{eq:root2}
\end{equation}
with $d_{\lambda_f}(g_1,h_1):=(g_1+\lambda_f h_1)/2$. Due to the fact that the square-root term is either positive or imaginary, the convergence rate is limited by the first root (with the $+$ sign). In addition, $f\in C_{\mu,L}$, hence $\lambda_f$ may vary between $1/\kappa$ and $1$. As a result, the convergence rate is lower bounded by
\begin{equation}
\min_{g_1,h_1} \max_{\lambda_f \in [\frac{1}{\kappa},1]} \text{Re}\left(-d_{\lambda_f} + \sqrt{d_{\lambda_f}^2 - \lambda_f}\right), \label{eq:minmax}
\end{equation}
where the dependence of $d_{\lambda_f}$ on $g_1$ and $h_1$ has been omitted. For a fixed $\lambda_f>0$ the function achieves its minimum value $-\sqrt{\lambda_f}$ when $g_1$ and $h_1$ are chosen such that $d_{\lambda_f}(g_1,h_1)=\sqrt{\lambda_f}$. Thus, interchanging the min and the max concludes that \eqref{eq:minmax} is lower bounded by $-1/\sqrt{\kappa}$. Choosing $g_1=2/\sqrt{\kappa}$, $h_1=0$, reveals that the lower bound is actually attained.\footnote{In that case, the real part of \eqref{eq:root2} evaluates to $-1/\sqrt{\kappa}$, whereas changing $\lambda_f\in [1/\kappa,1]$ only affects the imaginary part.} This implies
\begin{equation*}
\min_{g_1,h_1} \max_{\lambda_f \in [\frac{1}{\kappa},1]} \text{Re}\left(-d_{\lambda_f} + \sqrt{d_{\lambda_f}^2 - \lambda_f}\right)=-1/\sqrt{\kappa},
\end{equation*}
which, by virtue of Lemma~\ref{Lem:lin}, implies:

\emph{The convergence rate of any continuous-time optimization algorithm $(g,h)\in \mathcal{G}'$ (with $k=2$) is upper bounded by $1/\sqrt{\kappa}$ on functions $C_{\mu,L}$. Thus, the time required to achieve an $\epsilon$-distance to the optimizer is lower bounded by $\sqrt{\kappa}~\text{ln}(c/\epsilon)$ on functions $C_{\mu,L}$, where $c>0$ is constant.}

\subsection{Higher-order methods ($k>2$)}
We follow the reasoning of the previous sections and obtain the characteristic polynomial
\begin{multline}
s^{k}+ (g_{k-1}+h_{k-1} \lambda_f) s^{k-1} + \dots \\
+ (g_1+h_1 \lambda_f)s + \lambda_f=0, \label{eq:charPoly}
\end{multline}
whose roots determine the convergence rate of a single component of $x(t)$, where $x(t)$ satisfies \eqref{eq:first-orderNorm} and $\lambda_f \in [1/\kappa,1]$ is an eigenvalue of $\Delta f(0)/L$. 
Expressing \eqref{eq:charPoly} in terms of its roots $-\pi_1,\dots,-\pi_k\in \mathbb{C}$, where $\mathbb{C}$ denotes the set of complex numbers, results in
\begin{equation}
(s+\pi_1)(s+\pi_2)\dots (s+\pi_k)=0. \label{eq:factor}
\end{equation}
The minus sign is introduced for notational convenience. Equating the coefficients of \eqref{eq:charPoly} and \eqref{eq:factor} yields
\begin{align}
\lambda_f=\pi_1 \pi_2 \dots \pi_k. \label{eq:polecons}
\end{align}
In other words, no matter how the coefficient $g_j$ and $h_j$ are chosen, the product of the roots of \eqref{eq:charPoly} is always equal to $\lambda_f$. 
Due to the fact that the coefficients of the polynomial \eqref{eq:charPoly} are real, the roots $\pi_1,\dots,\pi_k$ are complex conjugated. As a consequence, the previous equation simplifies to
\begin{equation}
\lambda_f=|\pi_1|~|\pi_2| \dots |\pi_k| \geq |\pi_\text{min}|^k, \label{eq:eqsat}
\end{equation}
where $\pi_\text{min}$ denotes the the root with the smallest absolute value. Evaluating \eqref{eq:eqsat} for $\lambda_f=1/\kappa$ therefore yields the upper bound $1/\kappa^{1/k}$ on the absolute value of the smallest root, which suggests that the convergence rate is limited by $1/\kappa^{1/k}$. We will show next that such a convergence rate can, in fact, be achieved.

\begin{proposition}\label{Prop:LBk}
The convergence rate of any continuous-time optimization algorithm $(g,h)\in \mathcal{G}'$ with $k\geq 1$ is upper bounded by $1/\kappa^{1/k}$ on functions $C_{\mu,L}$. The algorithm given by \eqref{eq:first-orderNorm} with 
\begin{equation}\label{eq:fastAlg}
h_j=\kappa^{j/k} {k-1 \choose j}, \quad g_j=\frac{1}{\kappa^{(k-j)/k}} {k \choose j}-\frac{h_j}{\kappa},
\end{equation}
locally achieves the upper bound.
\end{proposition}
\begin{proof}
The previous discussion implies that $1/\kappa^{1/k}$ is indeed a lower bound. 
In case $g_j$ and $h_j$ are chosen according to \eqref{eq:fastAlg}, the characteristic polynomial \eqref{eq:charPoly} simplifies to 
\begin{equation*}
(s+1/\kappa^{1/k})^k
+ \bar{\lambda}_f (s \kappa^{1/k} + 1)^{k-1}=0, 
\end{equation*}
where $\bar{\lambda}_f:=\lambda_f-1/\kappa$. This can be further factorized to 
\begin{equation*}
(s+1/\kappa^{1/k})^{k-1} (s+1/\kappa^{1/k} + \bar{\lambda}_f \kappa^{(k-1)/k})=0,
\end{equation*}
which shows that there are $k-1$ real roots at $-1/\kappa^{1/k}$ and one real root at
\begin{equation*}
-\frac{1}{\kappa^{1/k}} - \bar{\lambda}_f \kappa^{(k-1)/k}
\end{equation*}
that depends on $\bar{\lambda}_f$. For $\bar{\lambda}_f\in [0,1-1/\kappa]$, this root takes values in $[-\kappa^{1-1/k},-1/\kappa^{1/k}]$. \qed
\end{proof}
The example above locally achieves a convergence rate of $1/\kappa^{1/k}$, but introduces a single real root that tends to $-\infty$ for large $\kappa$, when $\lambda_f>1/\kappa$. Such a fast root poses a problem for any explicit discretization scheme \citep{Nevanlinna}. We show that the convergence rate of any continuous-time optimization algorithm cannot exceed $\mathcal{O}(1/\sqrt{\kappa})$ when all roots of \eqref{eq:charPoly} are required to remain bounded for all $\kappa$. 

\begin{proposition}\label{Prop:BP}
The convergence rate of any continuous-time optimization algorithm $(g,h)\in \mathcal{G}'$ ($k\geq 2$), whose characteristic polynomial \eqref{eq:charPoly} has bounded roots, cannot exceed $\mathcal{O}(1/\sqrt{\kappa})$.
\end{proposition}
\begin{proof}
We consider the case where the roots are constrained to the unit disk (unit bound); the same arguments also apply for a bound greater than one. For the subsequent analysis it will be beneficial to rescale $g_j$ and $\lambda_f$ by introducing $\bar{g}_j:=g_j+h_j/\kappa$, $\bar{\lambda}_f=\lambda_f-1/\kappa$ such that \eqref{eq:charPoly} takes the form
\begin{multline}
s^k+\bar{g}_{k-1} s^{k-1} + \dots + \bar{g}_1 s + 1/\kappa \\
+\bar{\lambda}_f \left(h_{k-1} s^{k-1} + \dots + h_1 s + 1\right) =0, \label{eq:charPoly2}
\end{multline}
with $\bar{\lambda}_f \in [0,1-1/\kappa]$. 
In order to study the dependence of the roots of \eqref{eq:charPoly2} on $\bar{\lambda}_f$ we use the Nyquist criterion (see App.~\ref{App:Nyquist}), which provides a necessary and sufficient condition for \eqref{eq:charPoly2} to have all roots in the left-half complex plane for all $\bar{\lambda}_f \in [0,1-1/\kappa]$. The Nyquist criterion implies that if the roots of \eqref{eq:charPoly2} are all in the left-half complex plane, there is no $\bar{\lambda}_f\in [0,1-1/\kappa]$, such that the graph of the complex function $P: \mathbb{R} \rightarrow \mathbb{C}$,
\begin{equation}
P(\omega):=\left.\bar{\lambda}_{f} \frac{h_{k-1} s^{k-1} + \dots + h_1 s + 1}{s^{k} + \bar{g}_{k-1} s^{k-1} + \dots + \bar{g}_1 s + 1/\kappa}\right|_{s=i \omega}, \label{eq:transferFunction}
\end{equation}
passes through the point $-1$, where $i:=\sqrt{-1}$ denotes the imaginary unit. We will show that this condition cannot be fulfilled when the roots $\pi_j$ of \eqref{eq:charPoly2} are required to satisfy $|\pi_j|\in (1/\sqrt{\kappa},1]
$ for all $\kappa\geq 1$. To that end, the function $P(\omega)$ is rewritten as
\begin{equation}
\frac{\bar{\lambda}_f}{1-1/\kappa} \Big(\underbrace{\frac{(s+\bar{\pi}_1)\dots (s+\bar{\pi}_k)}{(s+\underline{\pi}_1) \dots (s+\underline{\pi}_k)}}_{:=H(s)} -1\Big)\Big\rvert_{s=i\omega}, \label{eq:P2}
\end{equation}
where $\underline{\pi}_j$ are the roots of \eqref{eq:charPoly2} for $\bar{\lambda}_f=0$ and $\bar{\pi}_j$ are the roots of \eqref{eq:charPoly2} for $\bar{\lambda}_f=1-1/\kappa$. These therefore satisfy (cf.\ \eqref{eq:eqsat})
\begin{equation}
\frac{1}{\kappa}=|\underline{\pi}_1|\dots |\underline{\pi}_k|, \quad 1=|\bar{\pi}_1|\dots |\bar{\pi}_k|.
\end{equation}
Combined with the requirement $|\pi_j|\leq 1$, the latter implies $|\bar{\pi}_j|=1$ for $j=1,2,\dots,k$, whereas the former implies that at least three roots, denoted by $\underline{\pi}_1,\underline{\pi}_2,\underline{\pi}_3$ tend to zero for $\kappa \rightarrow \infty$, since $|\underline{\pi}_j|>\mathcal{O}(1/\sqrt{\kappa})$. Next, we analyze the graph of $P(\omega)$, and start by considering the numerator and denominator of $H(i\omega)$ separately. The numerator of $H(i \omega)$ takes the value $1$ for $\omega=0$, is continuous in $\omega$, and its graph spirals outwards for $\omega\geq 0$ (since $\text{Re}(\bar{\pi}_j)>0$, the phase strictly increases for $\omega\geq 0$ until it reaches $90\cdot k$ degrees). All roots $\bar{\pi}_j$ satisfy $|\bar{\pi}_j|=1$. Hence, there exists a value $\omega_{\text{num}}>0$ such that the numerator of $H(i\omega)$ has a phase below, say, 10 degrees for all $\omega \in [0,\omega_{\text{num}}]$, and all $\kappa\geq 1$. The denominator of $H(i \omega)$ is likewise continuous in $\omega$, takes the value $1/\kappa$ for $\omega=0$, and spirals outwards for $\omega\geq 0$. However, since at least three of the roots $\underline{\pi}_j$ tend to zero for $\kappa \rightarrow \infty$, the denominator approaches (for large $\kappa$)
\begin{equation*}
s^3 (s+\underline{\pi}_4) \dots (s+\underline{\pi}_k)|_{s=i\omega},
\end{equation*}
which has a phase of more than $270$ degrees for small $\omega$ (due to the $s^3$ term). Thus, there exists a small value $0<\omega_{\text{denum}}\leq \omega_{\text{num}}$, such that for sufficiently large $\kappa$, the denominator of $H(i\omega)$ reaches a phase of more than $190$ degrees for $\omega \in [0,\omega_{\text{denum}}]$. The situation is illustrated in Fig.~\ref{Fig:illus}. The phase of $H(i\omega)$ is given by the difference between the phase of the numerator and the denominator, and therefore, for sufficiently large $\kappa$, there exists the value $\omega_\text{c}\in [0,\omega_{\text{denum}}]$, such that $H(i\omega_\text{c})$ has a phase of $-180$ degrees. Hence, according to \eqref{eq:P2}, $P(i \omega_\text{c})$ has likewise a phase of $-180$ degrees and $P(i \omega_\text{c}) <-1$. Thus, there exists a $\bar{\lambda}_f$ such that $P(i\omega_{\text{c}})=-1$, contradicting the condition established by the Nyquist criterion. \qed
\end{proof}

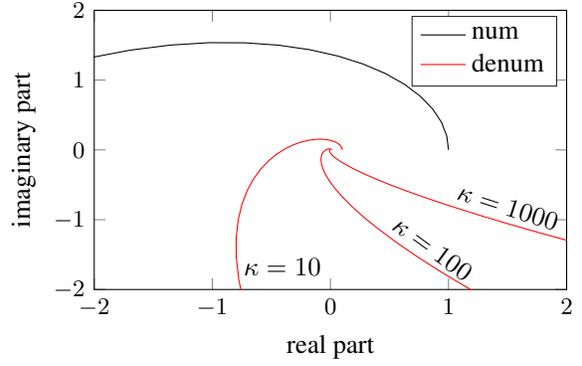
\begin{figure}
\newlength\figurewidth
\newlength\figureheight
\setlength{\figurewidth}{.8\columnwidth}
\setlength{\figureheight}{0.45\columnwidth}

\center
%
%
\begin{tikzpicture}

\begin{axis}[%
width=0.951\figurewidth,
height=\figureheight,
at={(0\figurewidth,0\figureheight)},
scale only axis,
xmin=-2,
xmax=2,
xlabel style={font=\color{white!15!black}},
xlabel={real part},
xlabel near ticks,
ymin=-2,
ymax=2,
ylabel style={font=\color{white!15!black}},
ylabel={imaginary part},
ylabel near ticks,
axis background/.style={fill=white},
legend style={legend cell align=left, align=left, draw=white!15!black}
]
\addplot [color=black]
  table[row sep=crcr]{%
1	0\\
0.98500625	0.1995\\
0.9401	0.396\\
0.86550625	0.5865\\
0.7616	0.768\\
0.62890625	0.9375\\
0.4681	1.092\\
0.28000625	1.2285\\
0.0655999999999998	1.344\\
-0.17399375	1.4355\\
-0.4375	1.5\\
-0.72349375	1.5345\\
-1.0304	1.536\\
-1.35649375	1.5015\\
-1.6999	1.428\\
-2.05859375	1.3125\\
-2.4304	1.152\\
-2.81299375	0.9435\\
-3.2039	0.684\\
-3.60049375	0.3705\\
-4	0\\
-4.39949375	-0.4305\\
-4.7959	-0.924000000000001\\
-5.18599375	-1.4835\\
-5.5664	-2.112\\
-5.93359375	-2.8125\\
-6.2839	-3.588\\
-6.61349375	-4.4415\\
-6.9184	-5.376\\
-7.19449375	-6.3945\\
-7.4375	-7.5\\
};
\addlegendentry{num}

\addplot [color=red]
  table[row sep=crcr]{%
0.1	0\\
0.0952628335097474	0.0352844175381833\\
0.0811263340389897	0.0688818111007955\\
0.0578155015877269	0.0991051567122657\\
0.0257053361559589	0.124267430397023\\
-0.0146791622563142	0.142681608179495\\
-0.0626629936490925	0.152660666084113\\
-0.117421158022376	0.152517580135304\\
-0.177978655376164	0.140565326357498\\
-0.243210485710458	0.115116880775124\\
-0.311841649025257	0.07448521941261\\
-0.382447145320561	0.0169833182943857\\
-0.45345197459637	-0.0590758465551202\\
-0.523131136852684	-0.155379299111479\\
-0.589609632089503	-0.273614063350261\\
-0.650862460306828	-0.415467163247037\\
-0.704714621504658	-0.582625622777379\\
-0.748841115682992	-0.776776465916859\\
-0.780766942841832	-0.999606716641045\\
-0.797867102981177	-1.25280339892551\\
-0.797366596101028	-1.53805353674583\\
-0.776340422201383	-1.85704415407756\\
-0.731713581282243	-2.21146227489629\\
-0.660261073343609	-2.60299492317758\\
-0.55860789838548	-3.03332912289701\\
-0.423229056407856	-3.50415189803014\\
-0.250449547410736	-4.01715027255255\\
-0.0364443713941218	-4.5740112704398\\
0.222761471641985	-5.17642191566747\\
0.531292981697589	-5.82606923221114\\
0.893425158772688	-6.52464024404636\\
};
\addlegendentry{denum}

\addplot [color=red, forget plot]
  table[row sep=crcr]{%
0.01	0\\
0.00850625	0.00616644143732834\\
0.0041	0.0113841995766062\\
-0.00299375	0.014704591119783\\
-0.0124	0.0151789327688082\\
-0.02359375	0.0118585412256314\\
-0.0359	0.00379473319220205\\
-0.04849375	-0.0099611746295304\\
-0.0604	-0.0303578655376165\\
-0.07049375	-0.0583440228301066\\
-0.0775	-0.0948683298050514\\
-0.07999375	-0.140879469760501\\
-0.0764	-0.197326125994507\\
-0.06499375	-0.265156981805119\\
-0.0439	-0.345320720490387\\
-0.01109375	-0.438766025348363\\
0.0355999999999999	-0.546441579677096\\
0.09850625	-0.669296066774637\\
0.1801	-0.808278169939037\\
0.28300625	-0.964336572468347\\
0.41	-1.13841995766062\\
0.56400625	-1.3314770088139\\
0.7481	-1.54445640922624\\
0.96550625	-1.77830684219569\\
1.2196	-2.0339769910203\\
1.51390625	-2.31241553899813\\
1.8521	-2.61457116942722\\
2.23800625	-2.94139256560562\\
2.6756	-3.29382841083138\\
3.16900625	-3.67282738840256\\
3.7225	-4.07933818161721\\
};
\addplot [color=red, forget plot]
  table[row sep=crcr]{%
0.001	0\\
0.000531908350974743	0.00103576867987875\\
-0.000797366596101028	0.00153805353674583\\
-0.00276282484122731	0.000973370747589547\\
-0.00498946638440411	-0.00119176351060176\\
-0.00695229122563142	-0.00549083306083978\\
-0.00797629936490925	-0.0124573217261362\\
-0.00723649080223758	-0.0226247133295026\\
-0.00375786553761643	-0.0365264916939508\\
0.0035845764289542	-0.0546961406424925\\
0.0160658350974743	-0.0776671439981392\\
0.035110910467944	-0.105972985583903\\
0.0622948025403631	-0.140147149222795\\
0.0993425113147316	-0.180723118737827\\
0.14812903679105	-0.22823437795201\\
0.210679378969317	-0.283214410688358\\
0.289168537849534	-0.34619670076988\\
0.385921513431701	-0.417714732019589\\
0.503413305715817	-0.498301988260497\\
0.644268914701882	-0.588491953315615\\
0.811263340389897	-0.688818111007955\\
1.00732158277986	-0.799813945160529\\
1.23551864187178	-0.922012939596347\\
1.49907951766564	-1.05594857813842\\
1.80137921016145	-1.20215434460977\\
2.14594271935921	-1.36116372283339\\
2.53644504525893	-1.53351019663231\\
2.97671118786059	-1.71972724982953\\
3.4707161471642	-1.92034836624806\\
4.02258492316976	-2.13590702971092\\
4.63659251587727	-2.36693672404112\\
};
\end{axis}
\node at (2.5,0.3) (tmp1){$\kappa=10$};
\node[rotate=-36] at (4.5,0.55) (tmp1){$\kappa=100$};
\node[rotate=-18] at (5.5,1.1) (tmp1){$\kappa=1000$};
\end{tikzpicture}%
\caption{The graph illustrates the behavior of the numerator and denumerator of $H(i \omega)$ for $\omega \geq 0$ on the example $(i\omega+1)^4/((i\omega+1/\kappa^{0.25})^4$. Both numerator and denumerator spiral outwards and their phase approaches $360$ degrees for large $\omega$. While the numerator (black) has a phase close to $0$ for small $\omega$, the phase of the denumerator (red) approaches $360$ degrees for large $\kappa$.}
\label{Fig:illus}
\end{figure}

\vspace{-11pt}\subsection{Discussion}
The analysis provides insights into the convergence limits of any gradient-based algorithm. It emphasizes the fact that these convergence limits result from limited curvature information, since by Assumption (C2), only upper and lower bounds on $\Delta f(0)$ are known (i.e. $\lambda_f\in [1/\kappa,1]$). Compared to the discrete-time case, where the convergence rate is upper bounded by $\mathcal{O}(1/\sqrt{\kappa})$, faster convergence rates can be achieved with continuous-time algorithms. These algorithms, however, necessarily include arbitrarily fast converging dynamics as shown by Prop.~\ref{Prop:BP} and cannot be discretized by explicit linear methods \citep{Nevanlinna}. This  recovers the classical discrete-time result. Even though the worst-case complexity in discrete-time cannot be improved, discretizing the dynamics \eqref{eq:fastAlg} with variable step-size schemes, or integration methods whose order increases with $\kappa$ might still lead to new discrete-time algorithms that achieve fast convergence rates. Additional background on the relation to discrete time is included in App.~\ref{App:BackgroundDiscretization}.


The proof of Prop.~\ref{Prop:BP} provides a graphical interpretation of the fact that a convergence rate of $\mathcal{O}(1/\sqrt{\kappa})$ cannot be improved (under the assumption of bounded rates): A faster convergence rate requires that at least three roots of \eqref{eq:charPoly} become arbitrarily small for large $\kappa$. These three roots introduce a negative phase shift of more than 180 degrees (for $\kappa\rightarrow \infty$ each introduces $-90$ degrees) in the function $P(i\omega)$, cf.\ \eqref{eq:transferFunction}, which leads to instability for large $\kappa$. Clearly, these issues do not occur for low-order dynamics with $k\leq 2$, as in that case, the phase shift is limited from the outset to $-180$ degrees.

One might suspect that the algorithm provided in Prop.~\ref{Prop:BP} is fragile, in the sense that it only achieves the rate of $1/\kappa^{1/k}$ on quadratic functions. The following proposition shows that this is not the case; the algorithm converges and achieves the rate $1/\kappa^{1/k}$ for all smooth and strongly convex functions and even for certain nonconvex functions.
\begin{proposition}\label{Prop:Conv}
Let $f\in C_{\mu,L}$ be such that
\begin{equation}
(\nabla f(x)/L-\alpha_\text{s} x)\T (\nabla f(x)/L-x) \leq 0, \label{eq:sector}
\end{equation}
holds for all $x\in \mathbb{R}^n$, where $\alpha_\text{s}>0$ is constant. Then, the origin is a globally asymptotically stable equilibrium for the dynamical system
\begin{multline}
x^{(k)}(t)=-g_{k-1} x^{(k-1)}(t) - \dots - g_1 \dot{x}(t) \\
- \nabla f(x(t) + h_{1} \dot{x}(t) +\dots + h_{k-1} x^{(k-1)}(t))/L, \label{eq:Lure}
\end{multline}
where $g_j$ and $h_j$ are defined in Prop.~\ref{Prop:LBk}. The trajectories converge with rate $1/\kappa^{1/k}$ if \eqref{eq:sector} holds for $\alpha_\text{s}=1/\kappa$.
\end{proposition}
The result can be proved by rewriting the dynamics as a Lur\'e problem with $\nabla f(x)/L-x/\kappa$ as the nonlinear feedback term and applying the circle criterion \citep[][p.~265]{Khalil}. A full proof can be found in App.~\ref{App:Prop3p3}. The assumption \eqref{eq:sector} ensures that $\nabla f(x)/L$ belongs to the sector $[\alpha_\text{s},1]$, as illustrated on a scalar example in Fig.~\ref{Fig:Res} (top left). The condition $\alpha_\text{s}=1/\kappa$ is fulfilled by smooth and strongly convex functions, which shows that the lower bound suggested by Prop.~\ref{Prop:LBk} is achieved. However, \eqref{eq:sector} does not require $f$ to be convex.

\begin{figure*}[h]
\setlength{\figurewidth}{.8\columnwidth}
\setlength{\figureheight}{0.3\columnwidth}
\begin{minipage}{\columnwidth}
\hspace*{.2cm}
%
%
\begin{tikzpicture}

\begin{axis}[%
width=0.951\figurewidth,
height=\figureheight,
at={(0\figurewidth,0\figureheight)},
scale only axis,
xmin=-4.5,
xmax=4.5,
xlabel style={font=\color{white!15!black}},
xlabel={$x$},
xlabel near ticks,
ymin=-4,
ymax=5,
ylabel style={font=\color{white!15!black}},
ylabel={$f$, $\nabla f$},
ylabel near ticks,
axis background/.style={fill=white},
legend style={legend cell align=left, align=left, draw=white!15!black},
legend pos=south east
]
\addplot [color=black]
  table[row sep=crcr]{%
-4.5	-0.794173220395152\\
-4.4	-0.771973903286178\\
-4.3	-0.749774586177204\\
-4.2	-0.72757526906823\\
-4.1	-0.705375951959256\\
-4	-0.683176634850282\\
-3.9	-0.660977317741308\\
-3.8	-0.638778000632334\\
-3.7	-0.616578683523361\\
-3.6	-0.594379366414386\\
-3.5	-0.572180049305413\\
-3.4	-0.630836218237527\\
-3.3	-0.689492387169642\\
-3.2	-0.748148556101757\\
-3.1	-0.806804725033872\\
-3	-0.865460893965986\\
-2.9	-0.924117062898101\\
-2.8	-0.982773231830216\\
-2.7	-1.04142940076233\\
-2.6	-1.10008556969445\\
-2.5	-1.15874173862656\\
-2.4	-1.07501955703898\\
-2.3	-0.991297375451391\\
-2.2	-0.907575193863807\\
-2.1	-0.823853012276222\\
-2	-0.740130830688638\\
-1.9	-0.656408649101054\\
-1.8	-0.572686467513469\\
-1.7	-0.488964285925885\\
-1.6	-0.405242104338301\\
-1.5	-0.321519922750717\\
-1.4	-0.323837684991751\\
-1.3	-0.326155447232785\\
-1.2	-0.328473209473819\\
-1.1	-0.330790971714853\\
-1	-0.333108733955887\\
-0.9	-0.335426496196922\\
-0.8	-0.337744258437956\\
-0.7	-0.34006202067899\\
-0.6	-0.342379782920024\\
-0.5	-0.344697545161058\\
-0.399999999999999	-0.275758036128846\\
-0.3	-0.206818527096635\\
-0.2	-0.137879018064423\\
-0.0999999999999996	-0.0689395090322114\\
0	0\\
0.0999999999999996	0.0689395090322114\\
0.2	0.137879018064423\\
0.3	0.206818527096635\\
0.399999999999999	0.275758036128846\\
0.5	0.344697545161058\\
0.6	0.397879116457121\\
0.7	0.451060687753184\\
0.8	0.504242259049248\\
0.9	0.557423830345311\\
1	0.610605401641374\\
1.1	0.663786972937437\\
1.2	0.7169685442335\\
1.3	0.770150115529563\\
1.4	0.823331686825626\\
1.5	0.876513258121689\\
1.6	0.880196855696278\\
1.7	0.883880453270867\\
1.8	0.887564050845455\\
1.9	0.891247648420044\\
2	0.894931245994633\\
2.1	0.898614843569221\\
2.2	0.90229844114381\\
2.3	0.905982038718399\\
2.4	0.909665636292987\\
2.5	0.913349233867576\\
2.6	0.87270933418282\\
2.7	0.832069434498065\\
2.8	0.791429534813309\\
2.9	0.750789635128554\\
3	0.710149735443798\\
3.1	0.669509835759042\\
3.2	0.628869936074287\\
3.3	0.588230036389531\\
3.4	0.547590136704776\\
3.5	0.50695023702002\\
3.6	0.594023467637758\\
3.7	0.681096698255495\\
3.8	0.768169928873233\\
3.9	0.85524315949097\\
4	0.942316390108708\\
4.1	1.02938962072645\\
4.2	1.11646285134418\\
4.3	1.20353608196192\\
4.4	1.29060931257966\\
4.5	1.3776825431974\\
};
\addlegendentry{$\nabla f$}

\addplot [color=red]
  table[row sep=crcr]{%
-4.5	2.70805147975106\\
-4.4	2.62974412356699\\
-4.3	2.55365669909382\\
-4.2	2.47978920633155\\
-4.1	2.40814164528018\\
-4	2.3387140159397\\
-3.9	2.27150631831012\\
-3.8	2.20651855239144\\
-3.7	2.14375071818365\\
-3.6	2.08320281568677\\
-3.5	2.02487484490078\\
-3.4	1.96472403152363\\
-3.3	1.89870760125327\\
-3.2	1.8268255540897\\
-3.1	1.74907789003292\\
-3	1.66546460908293\\
-2.9	1.57598571123972\\
-2.8	1.48064119650331\\
-2.7	1.37943106487368\\
-2.6	1.27235531635084\\
-2.5	1.15941395093479\\
-2.4	1.04772588615151\\
-2.3	0.944410039526995\\
-2.2	0.849466411061235\\
-2.1	0.762895000754233\\
-2	0.684695808605991\\
-1.9	0.614868834616506\\
-1.8	0.55341407878578\\
-1.7	0.500331541113812\\
-1.6	0.455621221600603\\
-1.5	0.419283120246152\\
-1.4	0.387015239859029\\
-1.3	0.354515583247802\\
-1.2	0.321784150412472\\
-1.1	0.288820941353038\\
-1	0.255625956069501\\
-0.9	0.222199194561861\\
-0.8	0.188540656830117\\
-0.7	0.154650342874269\\
-0.6	0.120528252694319\\
-0.5	0.0861743862902646\\
-0.399999999999999	0.0551516072257692\\
-0.3	0.0310227790644952\\
-0.2	0.0137879018064424\\
-0.0999999999999996	0.00344697545161056\\
0	0\\
0.0999999999999996	0.00344697545161056\\
0.2	0.0137879018064424\\
0.3	0.0310227790644952\\
0.399999999999999	0.0551516072257692\\
0.5	0.0861743862902646\\
0.6	0.123303219371173\\
0.7	0.165750209581689\\
0.8	0.21351535692181\\
0.9	0.266598661391538\\
1	0.325000122990873\\
1.1	0.388719741719813\\
1.2	0.45775751757836\\
1.3	0.532113450566513\\
1.4	0.611787540684273\\
1.5	0.696779787931638\\
1.6	0.784615293622536\\
1.7	0.872819159070894\\
1.8	0.96139138427671\\
1.9	1.05033196923999\\
2	1.13964091396072\\
2.1	1.22931821843891\\
2.2	1.31936388267456\\
2.3	1.40977790666767\\
2.4	1.50056029041824\\
2.5	1.59171103392627\\
2.6	1.68101396232879\\
2.7	1.76625290076284\\
2.8	1.8474278492284\\
2.9	1.9245388077255\\
3	1.99758577625411\\
3.1	2.06656875481426\\
3.2	2.13148774340592\\
3.3	2.19234274202911\\
3.4	2.24913375068383\\
3.5	2.30186076937007\\
3.6	2.35690945460296\\
3.7	2.42066546289762\\
3.8	2.49312879425406\\
3.9	2.57429944867227\\
4	2.66417742615225\\
4.1	2.76276272669401\\
4.2	2.87005535029754\\
4.3	2.98605529696284\\
4.4	3.11076256668992\\
4.5	3.24417715947878\\
};
\addlegendentry{$f$}

\addplot [color=black, dashed, forget plot]
  table[row sep=crcr]{%
-4.5	-4.5\\
-4.4	-4.4\\
-4.3	-4.3\\
-4.2	-4.2\\
-4.1	-4.1\\
-4	-4\\
-3.9	-3.9\\
-3.8	-3.8\\
-3.7	-3.7\\
-3.6	-3.6\\
-3.5	-3.5\\
-3.4	-3.4\\
-3.3	-3.3\\
-3.2	-3.2\\
-3.1	-3.1\\
-3	-3\\
-2.9	-2.9\\
-2.8	-2.8\\
-2.7	-2.7\\
-2.6	-2.6\\
-2.5	-2.5\\
-2.4	-2.4\\
-2.3	-2.3\\
-2.2	-2.2\\
-2.1	-2.1\\
-2	-2\\
-1.9	-1.9\\
-1.8	-1.8\\
-1.7	-1.7\\
-1.6	-1.6\\
-1.5	-1.5\\
-1.4	-1.4\\
-1.3	-1.3\\
-1.2	-1.2\\
-1.1	-1.1\\
-1	-1\\
-0.9	-0.9\\
-0.8	-0.8\\
-0.7	-0.7\\
-0.6	-0.6\\
-0.5	-0.5\\
-0.399999999999999	-0.399999999999999\\
-0.3	-0.3\\
-0.2	-0.2\\
-0.0999999999999996	-0.0999999999999996\\
0	0\\
0.0999999999999996	0.0999999999999996\\
0.2	0.2\\
0.3	0.3\\
0.399999999999999	0.399999999999999\\
0.5	0.5\\
0.6	0.6\\
0.7	0.7\\
0.8	0.8\\
0.9	0.9\\
1	1\\
1.1	1.1\\
1.2	1.2\\
1.3	1.3\\
1.4	1.4\\
1.5	1.5\\
1.6	1.6\\
1.7	1.7\\
1.8	1.8\\
1.9	1.9\\
2	2\\
2.1	2.1\\
2.2	2.2\\
2.3	2.3\\
2.4	2.4\\
2.5	2.5\\
2.6	2.6\\
2.7	2.7\\
2.8	2.8\\
2.9	2.9\\
3	3\\
3.1	3.1\\
3.2	3.2\\
3.3	3.3\\
3.4	3.4\\
3.5	3.5\\
3.6	3.6\\
3.7	3.7\\
3.8	3.8\\
3.9	3.9\\
4	4\\
4.1	4.1\\
4.2	4.2\\
4.3	4.3\\
4.4	4.4\\
4.5	4.5\\
};
\addplot [color=black, dashed, forget plot]
  table[row sep=crcr]{%
-4.5	-0.9\\
-4.4	-0.88\\
-4.3	-0.86\\
-4.2	-0.84\\
-4.1	-0.82\\
-4	-0.8\\
-3.9	-0.78\\
-3.8	-0.76\\
-3.7	-0.74\\
-3.6	-0.72\\
-3.5	-0.7\\
-3.4	-0.68\\
-3.3	-0.66\\
-3.2	-0.64\\
-3.1	-0.62\\
-3	-0.6\\
-2.9	-0.58\\
-2.8	-0.56\\
-2.7	-0.54\\
-2.6	-0.52\\
-2.5	-0.5\\
-2.4	-0.48\\
-2.3	-0.46\\
-2.2	-0.44\\
-2.1	-0.42\\
-2	-0.4\\
-1.9	-0.38\\
-1.8	-0.36\\
-1.7	-0.34\\
-1.6	-0.32\\
-1.5	-0.3\\
-1.4	-0.28\\
-1.3	-0.26\\
-1.2	-0.24\\
-1.1	-0.22\\
-1	-0.2\\
-0.9	-0.18\\
-0.8	-0.16\\
-0.7	-0.14\\
-0.6	-0.12\\
-0.5	-0.1\\
-0.399999999999999	-0.0799999999999999\\
-0.3	-0.06\\
-0.2	-0.04\\
-0.0999999999999996	-0.0199999999999999\\
0	0\\
0.0999999999999996	0.0199999999999999\\
0.2	0.04\\
0.3	0.06\\
0.399999999999999	0.0799999999999999\\
0.5	0.1\\
0.6	0.12\\
0.7	0.14\\
0.8	0.16\\
0.9	0.18\\
1	0.2\\
1.1	0.22\\
1.2	0.24\\
1.3	0.26\\
1.4	0.28\\
1.5	0.3\\
1.6	0.32\\
1.7	0.34\\
1.8	0.36\\
1.9	0.38\\
2	0.4\\
2.1	0.42\\
2.2	0.44\\
2.3	0.46\\
2.4	0.48\\
2.5	0.5\\
2.6	0.52\\
2.7	0.54\\
2.8	0.56\\
2.9	0.58\\
3	0.6\\
3.1	0.62\\
3.2	0.64\\
3.3	0.66\\
3.4	0.68\\
3.5	0.7\\
3.6	0.72\\
3.7	0.74\\
3.8	0.76\\
3.9	0.78\\
4	0.8\\
4.1	0.82\\
4.2	0.84\\
4.3	0.86\\
4.4	0.88\\
4.5	0.9\\
};
\end{axis}
\end{tikzpicture}%
\end{minipage}\hfill
\begin{minipage}{\columnwidth}
\hspace*{-.35cm}\input{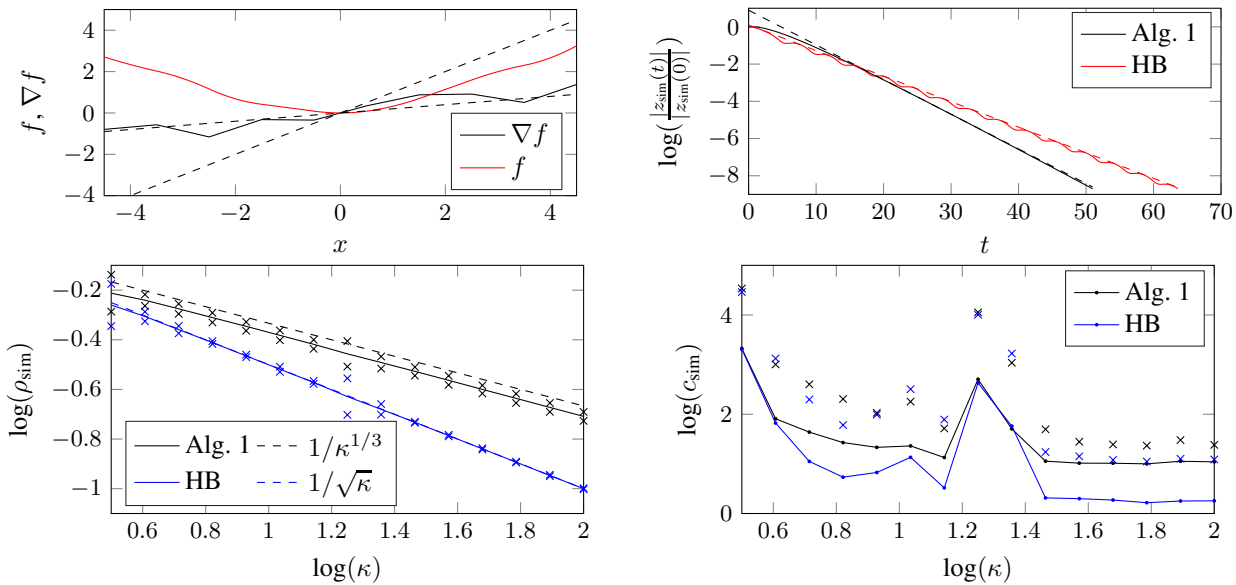}
\end{minipage}
\setlength{\figureheight}{0.4\columnwidth}
\begin{minipage}{\columnwidth}
%
%
\begin{tikzpicture}

\begin{axis}[%
width=0.951\figurewidth,
height=\figureheight,
at={(0\figurewidth,0\figureheight)},
scale only axis,
xmin=0.5,
xmax=2,
xlabel style={font=\color{white!15!black}},
xlabel={$\text{log}(\kappa)$},
ymin=-1.1,
ymax=-0.1,
ylabel style={font=\color{white!15!black}},
ylabel={$\text{log}(\rho_\text{sim})$},
xlabel near ticks,
ylabel near ticks,
axis background/.style={fill=white},
legend columns=2,
legend pos=south west,
legend style={legend cell align=left, align=left, draw=white!15!black}
]
\addplot [color=black]
  table[row sep=crcr]{%
0.5	-0.212241436972616\\
0.607142857142857	-0.240418674951509\\
0.714285714285714	-0.275264146522866\\
0.821428571428571	-0.310658137585194\\
0.928571428571429	-0.345824123501038\\
1.03571428571429	-0.38211571046629\\
1.14285714285714	-0.418238917778858\\
1.25	-0.456704900791315\\
1.35714285714286	-0.491611485361829\\
1.46428571428571	-0.527173002973896\\
1.57142857142857	-0.563352368811975\\
1.67857142857143	-0.599967154829205\\
1.78571428571429	-0.63603300964803\\
1.89285714285714	-0.672442027007023\\
2	-0.708707895140601\\
};
\addlegendentry{Alg.~1}

\addplot [color=black, dashed]
  table[row sep=crcr]{%
0.5	-0.166666666666667\\
0.607142857142857	-0.202380952380952\\
0.714285714285714	-0.238095238095238\\
0.821428571428571	-0.273809523809524\\
0.928571428571429	-0.30952380952381\\
1.03571428571429	-0.345238095238095\\
1.14285714285714	-0.380952380952381\\
1.25	-0.416666666666667\\
1.35714285714286	-0.452380952380952\\
1.46428571428571	-0.488095238095238\\
1.57142857142857	-0.523809523809524\\
1.67857142857143	-0.55952380952381\\
1.78571428571429	-0.595238095238095\\
1.89285714285714	-0.630952380952381\\
2	-0.666666666666667\\
};
\addlegendentry{$1/\kappa^{1/3}$}

\addplot [color=black, draw=none, mark=x, mark options={solid, black},forget plot]
  table[row sep=crcr]{%
0.5	-0.138010401925522\\
0.607142857142857	-0.217480680577045\\
0.714285714285714	-0.255100722220959\\
0.821428571428571	-0.29118004871588\\
0.928571428571429	-0.327993795668347\\
1.03571428571429	-0.362478880126117\\
1.14285714285714	-0.399913295952754\\
1.25	-0.405551514606925\\
1.35714285714286	-0.467330125779494\\
1.46428571428571	-0.509853187097502\\
1.57142857142857	-0.544903270437811\\
1.67857142857143	-0.583813073135915\\
1.78571428571429	-0.617482356919999\\
1.89285714285714	-0.654675227745856\\
2	-0.690600161021968\\
};

\addplot [color=black, draw=none, mark=x, mark options={solid, black},forget plot]
  table[row sep=crcr]{%
0.5	-0.286472472019711\\
0.607142857142857	-0.263356669325974\\
0.714285714285714	-0.295427570824774\\
0.821428571428571	-0.330136226454507\\
0.928571428571429	-0.363654451333729\\
1.03571428571429	-0.401752540806462\\
1.14285714285714	-0.436564539604961\\
1.25	-0.507858286975704\\
1.35714285714286	-0.515892844944165\\
1.46428571428571	-0.54449281885029\\
1.57142857142857	-0.581801467186139\\
1.67857142857143	-0.616121236522495\\
1.78571428571429	-0.654583662376062\\
1.89285714285714	-0.690208826268191\\
2	-0.726815629259234\\
};

\addplot [color=blue]
  table[row sep=crcr]{%
0.5	-0.25989313119509\\
0.607142857142857	-0.306081992291375\\
0.714285714285714	-0.35903903410971\\
0.821428571428571	-0.411171568930104\\
0.928571428571429	-0.464937043158625\\
1.03571428571429	-0.518675923857177\\
1.14285714285714	-0.571954049041203\\
1.25	-0.628945414559299\\
1.35714285714286	-0.680217400385309\\
1.46428571428571	-0.732414855430452\\
1.57142857142857	-0.785991966655694\\
1.67857142857143	-0.839594337325817\\
1.78571428571429	-0.892935856542317\\
1.89285714285714	-0.94676564146844\\
2	-1.00001566877871\\
};
\addlegendentry{HB}

\addplot [color=blue, draw=none, mark=x, mark options={solid, blue},forget plot]
  table[row sep=crcr]{%
0.5	-0.174949322283577\\
0.607142857142857	-0.286759522017637\\
0.714285714285714	-0.344054528608558\\
0.821428571428571	-0.405462213931511\\
0.928571428571429	-0.459968082841918\\
1.03571428571429	-0.508280122447192\\
1.14285714285714	-0.566019974312172\\
1.25	-0.555356183777473\\
1.35714285714286	-0.659086438972832\\
1.46428571428571	-0.730421877545664\\
1.57142857142857	-0.782742129674502\\
1.67857142857143	-0.836658407597569\\
1.78571428571429	-0.891282070301466\\
1.89285714285714	-0.943712738424125\\
2	-0.997917233417625\\
};

\addplot [color=blue, draw=none, mark=x, mark options={solid, blue},forget plot]
  table[row sep=crcr]{%
0.5	-0.344836940106604\\
0.607142857142857	-0.325404462565114\\
0.714285714285714	-0.374023539610861\\
0.821428571428571	-0.416880923928697\\
0.928571428571429	-0.469906003475332\\
1.03571428571429	-0.529071725267163\\
1.14285714285714	-0.577888123770235\\
1.25	-0.702534645341125\\
1.35714285714286	-0.701348361797785\\
1.46428571428571	-0.73440783331524\\
1.57142857142857	-0.789241803636887\\
1.67857142857143	-0.842530267054065\\
1.78571428571429	-0.894589642783169\\
1.89285714285714	-0.949818544512756\\
2	-1.00211410413979\\
};

\addplot [color=blue, dashed]
  table[row sep=crcr]{%
0.5	-0.25\\
0.607142857142857	-0.303571428571429\\
0.714285714285714	-0.357142857142857\\
0.821428571428571	-0.410714285714286\\
0.928571428571429	-0.464285714285714\\
1.03571428571429	-0.517857142857143\\
1.14285714285714	-0.571428571428571\\
1.25	-0.625\\
1.35714285714286	-0.678571428571429\\
1.46428571428571	-0.732142857142857\\
1.57142857142857	-0.785714285714286\\
1.67857142857143	-0.839285714285714\\
1.78571428571429	-0.892857142857143\\
1.89285714285714	-0.946428571428571\\
2	-1\\
};
\addlegendentry{$1/\sqrt{\kappa}$}

\end{axis}
\end{tikzpicture}%
\end{minipage}\hfill
\begin{minipage}{\columnwidth}
%
%
\begin{tikzpicture}

\begin{axis}[%
width=0.951\figurewidth,
height=\figureheight,
at={(0\figurewidth,0\figureheight)},
scale only axis,
xmin=0.5,
xmax=2,
xlabel style={font=\color{white!15!black}},
xlabel={$\text{log}(\kappa)$},
ymin=0,
ymax=5,
ylabel style={font=\color{white!15!black}},
ylabel={$\text{log}(c_\text{sim})$},
xlabel near ticks,
ylabel near ticks,
axis background/.style={fill=white},
legend style={legend cell align=left, align=left, draw=white!15!black}
]
\addplot [color=black,mark=*,mark size=.5pt]
  table[row sep=crcr]{%
0.5	3.32551077746658\\
0.607142857142857	1.90861327097059\\
0.714285714285714	1.63966683781579\\
0.821428571428571	1.42885341693048\\
0.928571428571429	1.33306243343818\\
1.03571428571429	1.36148285424442\\
1.14285714285714	1.12795115403245\\
1.25	2.70499167794527\\
1.35714285714286	1.70549750875492\\
1.46428571428571	1.05468396436752\\
1.57142857142857	1.01461152764347\\
1.67857142857143	1.01244114267616\\
1.78571428571429	1.00089301206024\\
1.89285714285714	1.05251177401512\\
2	1.04066913582938\\
};
\addlegendentry{Alg.~1}

\addplot [color=black, draw=none, mark=x, mark options={solid, black},forget plot]
  table[row sep=crcr]{%
0.5	4.53311112233382\\
0.607142857142857	3.00404147003167\\
0.714285714285714	2.60287727696309\\
0.821428571428571	2.30698252843754\\
0.928571428571429	2.02834142710791\\
1.03571428571429	2.2541554658307\\
1.14285714285714	1.71428518596467\\
1.25	4.05241838320877\\
1.35714285714286	3.03707327946664\\
1.46428571428571	1.69646250849025\\
1.57142857142857	1.44518894677256\\
1.67857142857143	1.39109799789736\\
1.78571428571429	1.36726090286156\\
1.89285714285714	1.47940003901863\\
2	1.38278580926933\\
};
\addplot [color=blue,mark=*,mark size=.5pt]
  table[row sep=crcr]{%
0.5	3.3100416598107\\
0.607142857142857	1.82297869560906\\
0.714285714285714	1.04960332494504\\
0.821428571428571	0.732343808449928\\
0.928571428571429	0.826879553933954\\
1.03571428571429	1.13158488365854\\
1.14285714285714	0.516658069162343\\
1.25	2.63068349616818\\
1.35714285714286	1.76072516102495\\
1.46428571428571	0.314541620787289\\
1.57142857142857	0.2989883576372\\
1.67857142857143	0.271350338091268\\
1.78571428571429	0.217917710612414\\
1.89285714285714	0.251718966438278\\
2	0.255273174381836\\
};

\addlegendentry{HB}

\addplot [color=blue, draw=none, mark=x, mark options={solid, blue},forget plot]
  table[row sep=crcr]{%
0.5	4.46185288305332\\
0.607142857142857	3.12330131363269\\
0.714285714285714	2.29670676236003\\
0.821428571428571	1.78306086442641\\
0.928571428571429	1.99109468159094\\
1.03571428571429	2.50501515336176\\
1.14285714285714	1.89303422097445\\
1.25	4.00420137226979\\
1.35714285714286	3.22470403229946\\
1.46428571428571	1.23799366452758\\
1.57142857142857	1.15164823117405\\
1.67857142857143	1.08013579954942\\
1.78571428571429	1.04709029373069\\
1.89285714285714	1.10232709648911\\
2	1.08905147587487\\
};
\end{axis}
\end{tikzpicture}%
\end{minipage}

\caption{Top left: The graph shows a randomly generated nonconvex function $f\in C_{\mu,1}$ in red. Its gradient (black, solid) is shown with the sector bound $[1/\kappa, 1]$ (black, dashed). Thus, $\nabla f$ satisfies the sector bound $[\alpha_\text{s},1]$ for some $\alpha_\text{s}>0$, but not the sector bound $[1/\kappa,1]$. Top right: The graph compares two trajectories obtained with Alg.~1 and Heavy Ball, respectively. The resulting convergence estimates $\rho_\text{sim}$ and $c_\text{sim}$ are indicated with dashed lines. Bottom left: The graph shows the convergence rate $\rho_\text{sim}$ when varying $\kappa$. Black relates to Alg.~1, blue to Heavy Ball. The solid lines show the mean across all initial conditions and all functions. The two-sigma bounds are marked with crosses. The dashed lines show the theoretical limits $1/\kappa^{1/3}$, respectively $1/\sqrt{\kappa}$. The mean for Heavy Ball lies almost exactly on the bound $1/\sqrt{\kappa}$. Bottom right: The graph shows the constant $c_\text{sim}$ for different $\kappa$. The solid lines indicate the mean across all initial conditions and all functions (black for Prop.~\ref{Prop:LBk} and blue for Heavy Ball). The crosses indicate the upper two-sigma bound.}
\label{Fig:Res}
\end{figure*}

\section{Simulation Results}\label{Sec:simRes}
This section illustrates the results from Sec.~\ref{Sec:LBsec} on a simulation example. We choose $C_{\mu,1}$ to be the set of all scalar functions that have a single non-degenerate minimum at the origin and whose Hessian is constant on the intervals $I_1=(-\infty, -4.5), I_2=(-4.5,-3.5),\dots, I_9=(4.5,\infty)$. The Hessian may change on the interval boundaries and takes the values summarized in Table~\ref{Tab:Hess}. Thus, the set $C_{\mu,1}$ includes certain nonconvex functions and satisfies the Assumptions (C1)-(C3). In addition, it is straightforward to generate random functions $f\in C_{\mu,1}$, by uniformly sampling potential values of the Hessian until the resulting function has a single local minimum.

\begin{table}
\begin{tabular}{c|c|c|c|c|c}
 & $I_1$ & $I_2$-$I_4$ & $I_5$ & $I_6-I_8$ & $I_9$\\ \hline
$\Delta f$ & $(0,1]$ & $[-1,1]$ &$[\mu,1]$ &$[-1,1]$ &$(0,1]$
\end{tabular}
\caption{The table summarizes the values that the Hessian can take. The interval $I_5$ ranges from $(-0.5,0.5)$, thus, the lower bound $\mu>0$ ensures that each $f$ has a non-degenerate local minimum at the origin. The Hessian is restricted to be positive on the intervals $I_1$ and $I_9$ in order to exclude any stationary point except the origin.}\label{Tab:Hess}
\end{table}

We evaluate and compare the performance of two algorithms. The first algorithm is that provided by Prop.~\ref{Prop:LBk}, for $k=3$, which we refer to as Alg.~1.
According to Prop.~\ref{Prop:Conv}, Alg.~1 is guaranteed to converge on functions $f\in C_{\mu,1}$, and its convergence rate is lower bounded by $1/\kappa^{1/3}$. The second algorithm is given by a variant of Heavy Ball:
\begin{equation}
\ddot{x}(t)=-2 \dot{x}(t)/\sqrt{\kappa} - \nabla f(x(t)). \label{eq:Alg2}
\end{equation}
The algorithm is guaranteed to converge on functions $f\in C_{\mu,1}$, since it dissipates energy as long as $|\dot{x}(t)|> 0$. According to Sec.~\ref{Sec:AccGradientFlow}, its convergence rate is upper bounded by $1/\sqrt{\kappa}$.

The trajectories are simulated with the standard fourth-order Runge-Kutta method with a time step of $0.01$. Both algorithms are evaluated on 50 randomly generated functions $f\in C_{\mu,1}$, and for each $f$, 50 simulations with randomized initial conditions are performed. The initial conditions are sampled from independent normal distributions with zero mean and standard deviation $4.5$ (motivated by the interval boundary $I_1$ and $I_9$). Each simulation terminates once a tolerance of $|z_\text{sim}(T_\text{max})|\leq 10^{-8}$ is reached, where $z_\text{sim}(t)$ denotes the simulated state trajectory that includes position, velocity, and, potentially, acceleration. The convergence rate $\rho_\text{sim}$ is estimated by performing a least-squares fit:
\begin{equation*}
\text{ln}(|z_\text{sim}(t)|/|z_\text{sim}(0)|) \approx  -\rho_\text{sim} t + \text{ln}(c),
\end{equation*}
with parameters $(\rho_\text{sim},\text{ln}(c))$. Only times $t>10$ are considered in order to avoid biases from fast-decaying transients. In a subsequent step, the smallest constant $c_\text{sim}$ satisfying 
\begin{equation*}
|z_\text{sim}(t)|\leq c_\text{sim} |z_\text{sim}(0)| e^{-\rho_\text{sim} t}, \quad \forall t\in [0,T_\text{max}]
\end{equation*}
is determined. The situation is illustrated on an example in Fig.~\ref{Fig:Res} (top right).\footnote{The estimation of $\rho_\text{sim}$ and $c_\text{sim}$ over the finite-time interval $t\in [0,T_\text{max}]$ is ill posed. For example, the constant $c_\text{sim}$ can be increased in favor of a better convergence rate $\rho_\text{sim}$. Nevertheless, the suggested least-squares procedure typically provides reliable estimates as shown in Fig.~\ref{Fig:Res} (top right). We also tested the procedure on quadratic functions, where the numerical results match the available closed-form expressions.}

The resulting values $\rho_\text{sim}$ and $c_\text{sim}$ for different $\kappa$ are summarized in Fig.~\ref{Fig:Res} (bottom row). The results indicate that the convergence rate of Heavy Ball roughly scales with $-1/\sqrt{\kappa}$, whereas the convergence rate of Alg.~1 scales with $-1/\kappa^{1/3}$ as suggested by Prop.~\ref{Prop:LBk} and Prop.~\ref{Prop:Conv}. Even though the convergence rate of Alg.~1 is superior, it tends to have higher constants $c_\text{sim}$. Also the variance in $\rho_\text{sim}$ and $c_\text{sim}$ seems higher. We observe that the obtained convergence rate estimates closely match the theoretical predictions.

\section{Conclusions}\label{Sec:Conclusion}
We have shown that the convergence rate of first-order optimization algorithms is lower bounded, not only in discretetime, but also in continuous time. The analysis shows that these limits are due to incomplete curvature information, since only upper and lower bounds on the local curvature of the objective function are assumed to be known. We found that the convergence rate of a $k$th-order algorithm is limited by $1/\kappa^{1/k}$ and provided an explicit algorithm that achieves this rate on smooth and strongly convex functions and even on certain nonconvex functions. We also note that this result is deceptive, since any algorithm whose convergence rate improves upon $\mathcal{O}(1/\sqrt{\kappa})$ necessarily includes dynamics that converge arbitrarily fast for large $\kappa$. Such an algorithm cannot be discretized with explicit linear methods. If such fast-converging dynamics are excluded, the analysis recovers the well-known asymptotic lower bound $\mathcal{O}(1/\sqrt{\kappa})$.

Numerical results with second and third-order dynamics indicate that the lower bound $1/\kappa^{1/3}$ is likely to be achieved even on nonconvex functions.

\section*{Acknowledgements}
We thank the Branco Weiss Fellowship, administered by ETH Zurich, for the generous support and the Office of Naval Research under grant number N00014-18-1-2764.
%
%

\bibliography{literature}

\begin{thebibliography}{21}
\providecommand{\natexlab}[1]{#1}
\providecommand{\url}[1]{\texttt{#1}}
\expandafter\ifx\csname urlstyle\endcsname\relax
  \providecommand{\doi}[1]{doi: #1}\else
  \providecommand{\doi}{doi: \begingroup \urlstyle{rm}\Url}\fi

\bibitem[Arjevani et~al.(2016)Arjevani, Shalev-Shwartz, and Shamir]{Arjevani}
Arjevani, Y., Shalev-Shwartz, S., and Shamir, O.
\newblock On lower and upper bounds in smooth and strongly convex optimization.
\newblock \emph{Journal of Machine Learning Research}, 17\penalty0
  (126):\penalty0 1--51, 2016.

\bibitem[{\AA}str{\"o}m \& Murray(2008){\AA}str{\"o}m and Murray]{Astrom}
{\AA}str{\"o}m, K.~J. and Murray, R.~M.
\newblock \emph{Feedback Systems}.
\newblock Princeton University Press, second edition, 2008.

\bibitem[Bellman(1953)]{bellmanODE}
Bellman, R.
\newblock \emph{Stability Theory of Differential Equations}.
\newblock McGraw-Hill, 1953.

\bibitem[Butcher(2016)]{Butcher}
Butcher, J.~C.
\newblock \emph{Numerical Methods for Ordinary Differential Equations}.
\newblock Wiley, third edition, 2016.

\bibitem[Callier \& Desoer(1991)Callier and Desoer]{Desoer}
Callier, F.~M. and Desoer, C.~A.
\newblock \emph{Linear System Theory}.
\newblock Springer Science and Business Media, 1991.

\bibitem[Carmon et~al.(2019)Carmon, Duchi, Hinder, and Sidford]{Duchi}
Carmon, Y., Duchi, J.~C., Hinder, O., and Sidford, A.
\newblock Lower bounds for finding stationary points {II}: first-order methods.
\newblock \emph{Mathematical Programming}, 2019.
\newblock Preprint available online.

\bibitem[Diakonikolas \& Jordan(2019)Diakonikolas and Jordan]{Diakonakolis}
Diakonikolas, J. and Jordan, M.~I.
\newblock Generalized momentum-based methods: {A} {H}amiltonian perspective.
\newblock \emph{arXiv:1906.00436 [math.OC]}, pp.\  1--30, 2019.

\bibitem[Hahn(1967)]{Hahn}
Hahn, W.
\newblock \emph{Stability of Motion}.
\newblock Springer, 1967.

\bibitem[Hairer et~al.(1993)Hairer, N{\o}rsett, and Wanner]{Hairer}
Hairer, E., N{\o}rsett, S.~P., and Wanner, G.
\newblock \emph{Solving Ordinary Differential Equations}.
\newblock Springer, second edition, 1993.

\bibitem[Jin et~al.(2019)Jin, Netrapalli, Ge, Kakade, and Jordan]{Chi}
Jin, C., Netrapalli, P., Ge, R., Kakade, S.~M., and Jordan, M.~I.
\newblock On nonconvex optimization for machine learning: Gradients,
  stochasticity, and saddle points.
\newblock \emph{arXiv:1902.04811 [cs.LG]}, pp.\  1--31, 2019.

\bibitem[Khalil(1996)]{Khalil}
Khalil, H.~K.
\newblock \emph{Nonlinear Systems}.
\newblock Prentice-Hall, third edition, 1996.

\bibitem[Krichene et~al.(2015)Krichene, Bayen, and Bartlett]{KricheneAcc}
Krichene, W., Bayen, A.~M., and Bartlett, P.~L.
\newblock Accelerated mirror descent in continuous and discrete time.
\newblock \emph{Advances in Neural Information Processing Systems 28}, pp.\
  2845--2853, 2015.

\bibitem[Kunimatsu et~al.(2008)Kunimatsu, Sang-Hoon, Fujii, and Ishitobi]{PRL}
Kunimatsu, S., Sang-Hoon, K., Fujii, T., and Ishitobi, M.
\newblock On positive real lemma for non-minimal realization systems.
\newblock \emph{Proceedings of the 17th World Congress of the International
  Federation of Automatic Control}, pp.\  5868--5873, 2008.

\bibitem[Minnichelli et~al.(1989)Minnichelli, Anagost, and Desoer]{Kharitonov}
Minnichelli, R.~J., Anagost, J.~J., and Desoer, C.~A.
\newblock An elementary proof of {K}haritonov's stability theorem with
  extensions.
\newblock \emph{IEEE Transactions on Automatic Control}, 34\penalty0
  (9):\penalty0 995--998, 1989.

\bibitem[Muehlebach \& Jordan(2019)Muehlebach and Jordan]{MuehlebachJordan}
Muehlebach, M. and Jordan, M.~I.
\newblock A dynamical systems perspective on {N}esterov acceleration.
\newblock \emph{Proceedings of the International Conference on Machine
  Learning}, pp.\  1--7, 2019.

\bibitem[Nesterov(2004)]{NesterovBook}
Nesterov, Y.
\newblock \emph{Introductory Lectures on Convex Optimization}.
\newblock Springer, 2004.

\bibitem[Nevanlinna \& Sipil\"{a}(1974)Nevanlinna and Sipil\"{a}]{Nevanlinna}
Nevanlinna, O. and Sipil\"{a}, A.~H.
\newblock A nonexistence theorem for explicit {$A$}-stable methods.
\newblock \emph{Mathematics of Computation}, 28\penalty0 (128):\penalty0
  1053--1055, 1974.

\bibitem[Sastry(1999)]{SastryNonlinear}
Sastry, S.
\newblock \emph{Nonlinear Systems}.
\newblock Springer, 1999.

\bibitem[Scoy et~al.(2018)Scoy, Freeman, and Lynch]{TripleMomentum}
Scoy, B.~V., Freeman, R.~A., and Lynch, K.~M.
\newblock The fastest known globally convergent first-order method for
  minimizing strongly convex functions.
\newblock \emph{IEEE Control Systems Letters}, 2\penalty0 (1):\penalty0 49--54,
  2018.

\bibitem[Su et~al.(2016)Su, Boyd, and Cand\`{e}s]{SuAcc}
Su, W., Boyd, S., and Cand\`{e}s, E.~J.
\newblock A differential equation for modeling {N}esterov's accelerated
  gradient method: Theory and insights.
\newblock \emph{Journal of Machine Learning Research}, 17\penalty0
  (153):\penalty0 1--43, 2016.

\bibitem[Wibisono et~al.(2016)Wibisono, Wilson, and
  Jordan]{WibisonoVariational}
Wibisono, A., Wilson, A.~C., and Jordan, M.~I.
\newblock A variational perspective on accelerated methods in optimization.
\newblock \emph{Proceedings of the National Academy of Sciences}, 113\penalty0
  (47):\penalty0 E7351--E7358, 2016.

\end{thebibliography}
\bibliographystyle{icml2020}

\newpage
\appendix
\section{Interpretation of \eqref{eq:timenorm}}\label{App:TimeNorm}
In a neighborhood of the origin, the nonlinear dynamics \eqref{eq:ode} are closely approximated by the corresponding linear dynamics \eqref{eq:lineq} (with $\tilde{r}=0$). Thus, under mild assumptions (exponential stability and distinct eigenvalues of $A$) and for large $t$, the trajectories $x(t)$ of \eqref{eq:ode} can be approximated by \citep[p.~50]{bellmanODE}
\begin{equation*}
    x(t)\approx \sum_{j=1}^{nk} c_{j} \exp(\pi_{j} t),
\end{equation*}
where $c_{j} \in \mathbb{C}^n$ are constants and $\pi_j$ are the eigenvalues of the matrix $A$ in \eqref{eq:lineq}. These are complex conjugated and have a negative real part, which characterizes the (asymptotic) convergence rate. The time-normalization constraint \eqref{eq:timenorm} has the following interpretation:
\begin{proposition}
The time-normalization constraint \eqref{eq:timenorm} fixes the geometric mean of the eigenvalues $\pi_j$ for functions of the class $C_{L,L}$, that is,
\begin{equation*}
    \prod_{j=1}^{nk} |\pi_j| = 1.
\end{equation*}
This implies that the convergence rate is necessarily bounded by unity for functions of the class $C_{L,L}$.
\end{proposition}
\begin{proof}
The matrix $A$ is a companion matrix and therefore
\begin{equation*}
    \text{det}(-A)=\text{det}\left(-\left.\frac{\partial g(w,v)}{\partial v}\right|_{0,0} \Delta f(0) \left.\frac{\partial h}{\partial x}\right|_{0} \right),
\end{equation*}
where we exploited the fact that the partial derivative of $g$ with respect to $x$ vanishes, when evaluated at the origin, due to Assumption (G2). For functions of the class $C_{L,L}$, $\Delta f(0)$ reduces to $L$ times the identity. It follows that
\begin{equation*}
     \text{det}(-A) = (-L)^n \text{det}\left(\left.\frac{\partial g(w,v)}{\partial v}\right|_{0,0} \left.\frac{\partial h}{\partial x}\right|_{0} \right)=1,
\end{equation*}
where the time normalization constraint \eqref{eq:timenorm} has been used for the last equality.
The eigenvalues of $A$ are complex conjugated and have a negative real part, which concludes the first part of the proof: 
\begin{equation*}
    1=\text{det}(-A)=\prod_{j=1}^{nk} (-\pi_j)=\prod_{j=1}^{nk} |\pi_j|.
\end{equation*}
In addition, 
\begin{equation*}
    1= \prod_{j=1}^{nk} |\pi_j| \geq \left(\min_{j\in \{1,\dots,nk\}} |\pi_j|\right)^{nk},
\end{equation*}
which bounds the convergence rate for functions of the class $C_{L,L}$ by unity. \qed
\end{proof}

\section{Proof of Lemma~\ref{Lem:lin}}\label{App:Lemma}
In a neighborhood of the origin, exponential convergence of the nonlinear dynamics implies exponential convergence of the linearized dynamics.\footnote{This can be shown, for example, with the converse Lyapunov theorem from \citet[p.~195]{SastryNonlinear}.} In addition, \eqref{eq:lineq} relates the solutions of the nonlinear dynamics $z(t)$ to solutions of the linearized dynamics $\Delta z(t)$:
\begin{equation}
z(t)=\Delta z(t) + \int_{0}^{t} \exp(A (t-\tau)) \tilde{r}(z(\tau)) \diff \tau, \label{eq:linsol}
\end{equation}
with $\Delta z(0)=z(0)$ and where $\exp$ denotes the matrix exponential. We now exploit the fact that in a neighborhood of the origin, both $z(t)$ and $\Delta z(t)$ converge to the origin for $t\rightarrow \infty$, which concludes
\begin{equation*}
\int_{0}^{\infty}  \exp(A (t-\tau)) \tilde{r}(z(\tau)) \diff \tau = 0.
\end{equation*}
Rearranging \eqref{eq:linsol} therefore results in
\begin{equation*}
\Delta z(t)=z(t) + \int_{t}^{\infty} \exp(A (t-\tau)) \tilde{r}(z(\tau)) \diff \tau.
\end{equation*}
Moreover, $\exp(A (t-\tau)$ is bounded by $c_3 \exp(-\alpha (t-\tau))$, for a small enough $\alpha >0$ and a constant $c_3>0$, due to the fact that the linearized dynamics converge exponentially. Combined with the fact that $\tilde{r}$ is of second order, we obtain
\begin{multline*}
|\Delta z(t)| \leq c_1 |z(0)| \exp(-a_\text{r} t) + \\
c_3 c_4 c_1^2 |z(0)|^2 \exp(-\alpha t) \int_{t}^{\infty} \exp((\alpha -2 a_\text{r}) \tau) \diff \tau,
\end{multline*}
where $c_4>0$ is constant.
Without loss of generality, we can assume $\alpha< a_\text{r}$, which yields
\begin{multline*}
|\Delta z(t)| \leq c_1 |z(0)| \exp(-a_\text{r} t) + \\
c_3 c_4 c_1^2 |z(0)|^2 \exp(-2 a_\text{r} t)/(2 a_\text{r} - \alpha).
\end{multline*}
The fact that $|z(0)|=|\Delta z(0)|$ is bounded concludes the proof.

\section{Interpretation of Prop.~\ref{Prop:Inv}}\label{App:InterInv}
Prop.~\ref{Prop:Inv} has the following game-theoretic interpretation: We assume that Player I, who plays first, chooses the functions $(g,h)\in \mathcal{G}'$, which corresponds to the upper block in Fig.~\ref{Fig:Feedback}. Player II, who plays second, chooses the function $f\in C_{\mu,L}$, which defines the lower block in Fig.~\ref{Fig:Feedback}. Player I attempts to maximize the convergence rate, whereas Player II attempts to minimize the convergence rate. The statement of Prop.~\ref{Prop:Inv} implies that Player I does not loose anything by choosing dynamics that are invariant under orthogonal transformations (at least in the first-order approximation). If the dynamics are invariant under orthogonal transformations, a potential rotation of the function $f$ by Player II has no effect on the convergence rate, which provides an intuitive explanation of why this is a good choice for Player I.

\section{Background on the Nyquist criterion}\label{App:Nyquist}
This section provides additional details on the application of the Nyquist criterion in the proof of Prop.~\ref{Prop:BP}. We would like to emphasize that a complete treatment is beyond the scope of this article, and refer to the excellent texts by \citet{Astrom}, \citet{Desoer}, or \citet{Hahn} for further details. 

We start by rewriting the dynamics \eqref{eq:first-orderNorm} for a single component $x_j$ of $x$ in the following way:
\begin{gather}
x_j^{(k)} + \bar{g}_{k-1} x_j^{(k-1)} + \dots + \bar{g}_1 \dot{x}_j+ x_j/\kappa =u, \nonumber\\
y=x_j+h_1 \dot{x}_j + \dots + h_{k-1} x^{(k-1)}, \label{eq:feedback}
\end{gather}
with $u=-\bar{\lambda}_f y$, and where the coordinate frame has been chosen to diagonalize $\Delta f(0)$. As before, $\lambda_f \in [1/\kappa,1]$ denotes an eigenvalue of $\Delta f(0)/L$ and $\bar{\lambda}_f=\lambda_f-1/\kappa$. The dynamics \eqref{eq:first-orderNorm} are therefore interpreted as a feedback system according to Fig.~\ref{Fig:Feedback2}, where $P_\text{T}$ describes the relation between $y$ and $u$ according to \eqref{eq:feedback}. The relation from $y$ to $u$ can be described by the corresponding transfer function $P_\text{T}: \mathbb{C} \rightarrow \mathbb{C}$
\begin{equation}
P_\text{T}(s)=\frac{h_{k-1} s^{k-1} + \dots + h_1 s + 1}{s^{k} + \bar{g}_{k-1} s^{k-1} + \dots + \bar{g}_1 s + 1/\kappa},
\end{equation}
which is obtained by performing the Laplace transform on \eqref{eq:feedback} and eliminating the state variable $x_j$. We note that $\bar{\lambda}_f P_\text{T}(i\omega) = P(\omega)$ for all $\omega\in \mathbb{R}$, where $P(\omega)$ is defined in \eqref{eq:charPoly2}.

The Nyquist criterion provides a necessary and sufficient condition for the stability of the closed-loop system (the interconnection according to Fig.~\ref{Fig:Feedback2}) by means of the open-loop transfer function $\bar{\lambda}_f P_\text{T}$:

\begin{theorem}\label{Thm:Nyquist}
(Simplified version of the Nyquist criterion) Let the system $P_\text{T}$ be asymptotically stable in the sense of Lyapunov (for $u=0$). Then, the closed-loop system given by the interconnection according to Fig.~\ref{Fig:Feedback2} is asymptotically stable in the sense of Lyapunov if and only if the graph of $\bar{\lambda}_f P_\text{T}(i \omega)$ for $\omega \in (-\infty, \infty)$ does not encircle the point $-1$.
\end{theorem}
The statement of the Nyquist theorem is illustrated in Fig.~\ref{fig:NyquistExample}. 

The proof relies on the following observations: 1) The asymptotic stability of the closed-loop system is equivalent to the condition that the rational function $1 + \bar{\lambda}_f P_\text{T}(s)$ does not have any zeros in the closed right-half complex plane. 2) The poles of $1+ \bar{\lambda}_f P_\text{T}(s)$ are the same as the poles of $\bar{\lambda}_f P_\text{T}(s)$. Therefore, by assumption, $\bar{\lambda}_f P_\text{T}(s)$ and $1+\bar{\lambda}_f P_\text{T}(s)$ do not have any poles in the closed right-half complex plane. 3) We consider the Nyquist contour in the complex plane consisting of a vertical segment $[-iR,iR]$ along the imaginary axis and a right halfcircle of radius $R$ centered at the origin. For large enough $R$, the contour essentially encircles the right-half complex plane. Cauchy's argument principle can then be used to relate the argument of $1+\bar{\lambda}_f P_\text{T}(s)$, i.e. the encirclement of $-1$ by $\bar{\lambda}_f P_\text{T}(s)$, to the number of zeros of $1+\bar{\lambda}_f P_\text{T}(s)$ in the closed right-half complex plane. A detailed proof of a more general version of the theorem can be found in \citet[p.~368]{Desoer}.

The importance of the Nyquist theorem for the study of dynamical systems cannot be emphasized enough. One of the reasons is that the distance of the graph of $\bar{\lambda}_f P_\text{T}(i\omega)$ to the point $-1$ provides a meaningful characterization of how close the closed-loop system is to the boundary of stability. Another important feature is that the graph of $\bar{\lambda}_f P_\text{T}(i\omega)$ can often be obtained from input-output measurements alone, and that the theorem generalizes to systems with delays and certain nonlinearities. Finally, it is evident from the sketch of the proof that by shifting the Nyquist contour, the Nyquist theorem can also be used to derive convergence rates: Given that the system $P_\text{T}$ converges with rate $\rho>0$, the closed-loop system converges likewise with rate $\rho$ if and only if the graph of $\bar{\lambda}_f P_\text{T}(-\rho + i\omega)$ for $\omega \in (-\infty,\infty)$ does not encircle the point $-1$.

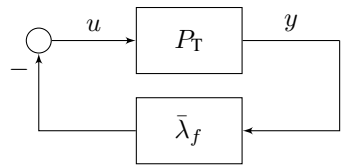
\begin{figure}
    \centering
    \tikzstyle{block} = [draw, rectangle, 
    minimum height=2.5em, minimum width=4em]
\tikzstyle{sum} = [draw, circle, node distance=.8cm]
\tikzstyle{input} = [coordinate]
\tikzstyle{output} = [coordinate]
\tikzstyle{pinstyle} = [pin edge={to-,thin,black}]

\begin{tikzpicture}[auto, node distance=2cm,>=latex']
    \node [block] (system) {$P_\text{T}$};
            								 
    \node [block, below of=system,node distance=1.2cm] (grad) {$\bar{\lambda}_f$};

    \node [sum, left of=system, node distance=2cm] (ln) {};
 	\node (rn) [coordinate,right of=system, node distance=2cm] {};
 	\draw [-] (system) -- node[name=y]{$y$} (rn);   
 	\draw [->] (rn) |- (grad); 
	\draw [->] (grad) -| node [pos=0.9] {$-$} (ln);
	\draw [->] (ln) -- node[name=u]{$u$} (system);
\end{tikzpicture}
    \caption{Interpretation of \eqref{eq:first-orderNorm} as feedback system as done in Prop.~\ref{Prop:BP}. The system $P_\text{T}$ in the upper block is described by \eqref{eq:feedback} and the lower block represents the multiplication by the scalar $\bar{\lambda}_f \in [0,1-1/\kappa]$.}
    \label{Fig:Feedback2}
\end{figure}

\begin{figure}
    \setlength{\figurewidth}{.8\columnwidth}
    \setlength{\figureheight}{0.45\columnwidth}
    \centering
    \input{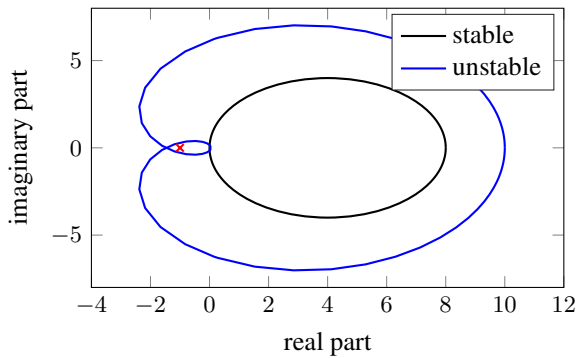}
    \caption{The figure illustrates the statement of the Nyquist criterion by showing the graphs of two open-loop transfer functions evaluated along the imaginary axis. The first graph (black) does not encircle the point $-1$ leading to a stable closed-loop system, whereas the second graph (blue) encircles the point $-1$, which leads to an unstable closed-loop system.}
    \label{fig:NyquistExample}
\end{figure}

\section{The relation to discrete time}\label{App:BackgroundDiscretization}
This section provides additional context on the discretization of \eqref{eq:ode} and the result of \citet{Nevanlinna}. Excellent books on the subject are \citet{Hairer} and \citet{Butcher}. We will concentrate on explicit discretization methods, as these solely rely on evaluating the right-hand side of \eqref{eq:ode} and do not require the solution of a system of equations at each time step. 

The two commonly used classes of numerical integration methods are linear multistage and linear multistep methods. Explicit linear multistage methods compute the next iterate as a linear combination of the past iterate and different evaluations of the right-hand side of \eqref{eq:ode} at predefined intermediate locations. In contrast, explicit linear multistep methods computed the next iterate as a linear combination of \emph{several} past iterates and the corresponding evaluations of the right-hand side of \eqref{eq:ode}. Explicit general linear methods, as considered in \citet{Nevanlinna}, combine both ideas, and include evaluations of \eqref{eq:ode} at intermediate locations, as well as multiple past iterates (and potentially also higher derivatives).

A fundamental requirement for the discretization of \eqref{eq:ode} is that the resulting discrete trajectories should converge linearly. Similar to the reasoning in Sec.~\ref{Sec:LBsec}, we can argue that this necessarily implies that the discretization of the corresponding linearized dynamics \eqref{eq:lineq} must be asymptotically stable. This reduces the problem to studying the discretization of the model equation $\dot{x}_\text{m}(t)=\lambda_\text{m} x_\text{m}(t)$, for different $\lambda_\text{m}\in \mathbb{C}$. In our case, $\lambda_\text{m}$ corresponds to the roots of the characteristic polynomial \eqref{eq:charPoly}. The set of all $\lambda_\text{m}\in \mathbb{C}$, $\text{real}(\lambda_\text{m})\leq 0$, such that the resulting discretization of $\dot{x}_\text{m}=\lambda_\text{m} x_\text{m}$ is stable is called the stability region of the numerical integration scheme. It has been shown in \citet{Nevanlinna} that the stability region of any explicit general linear method is necessarily bounded. In view of Prop.~\ref{Prop:BP}, this  implies that the discretization of any continuous-time algorithm that achieves a rate exceeding $\mathcal{O}(1/\sqrt{\kappa})$ with a general linear method will fail to even converge at a linear rate. This precludes any discrete algorithm that arises from the discretization of \eqref{eq:ode} with a general linear method, and achieves a convergence rate faster than $\mathcal{O}(1/\sqrt{\kappa})$.

The class of general linear methods considered in \citet{Nevanlinna} is large, and includes Runge-Kutta methods, multistep methods, multistep methods using higher derivatives, and predictor-corrector formulas. It is also important to notice that nonlinear or implicit discretization schemes, which are excluded in the analysis, typically require a much higher computational effort at each iteration.



\begin{figure*}
\mbox{}
\hfill %
\begin{minipage}{.6\columnwidth}
    \setlength{\figurewidth}{.9\columnwidth}
    \setlength{\figureheight}{0.7\columnwidth}
    \centering
%
%
\begin{tikzpicture}

\begin{axis}[%
width=0.951\figurewidth,
height=\figureheight,
at={(0\figurewidth,0\figureheight)},
scale only axis,
xmin=-3.43333333333333,
xmax=4.1,
xlabel style={font=\color{white!15!black}},
xlabel={real part},
xlabel near ticks,
ymin=-2.89193548387097,
ymax=2.89193548387097,
ylabel style={font=\color{white!15!black}},
ylabel={imaginary part},
ylabel near ticks,
legend style={legend cell align=left, align=left, draw=white!15!black},
disabledatascaling,
axis equal
]
\begin{pgfonlayer}{background}
  \draw[fill=black!20!white, draw=black!20!white] (axis cs:-2.33333333333333,0) circle (1);
\end{pgfonlayer}
\addplot [color=black,forget plot]
  table[row sep=crcr]{%
4	0\\
4	0\\
4	0\\
4	0\\
4	0\\
4	0\\
4	0\\
4	-6.21724893790088e-15\\
4	-6.30606677987089e-14\\
4	-6.3504757008559e-13\\
4	-6.3495875224362e-12\\
4	-7.99982302623903e-12\\
4	-6.3495875224362e-11\\
4	-6.34960528600459e-10\\
4	-6.34960439782617e-09\\
4	-6.34960422019049e-08\\
4	-7.99999995138023e-08\\
3.9999999999999	-6.34960420242681e-07\\
3.99999999998992	-6.34960420774532e-06\\
3.99999999899206	-6.34960420624523e-05\\
3.9999999984	-7.99999999676445e-05\\
3.99999989920632	-0.000634960404786968\\
3.999989920657	-0.00634958820791304\\
3.999984000064	-0.00799996800012817\\
3.99899231708027	-0.0634800461094593\\
3.9984006397441	-0.079968012794882\\
3.99793317649804	-0.0909011674757181\\
3.99732917370828	-0.103325562441295\\
3.99654881118492	-0.117442941703917\\
3.9955406965515	-0.133481565793814\\
3.99423852961729	-0.151699330881392\\
3.99255682359406	-0.17238707825341\\
3.9903854287711	-0.195872011619\\
3.98758254397725	-0.22252107962381\\
3.98396583240476	-0.252744091623315\\
3.9793011844799	-0.286996204010551\\
3.97328859925838	-0.325779225913639\\
3.96554460286519	-0.369640925964027\\
3.95558061614469	-0.419171150914693\\
3.94277677257346	-0.474993065158917\\
3.92635095086807	-0.537747165580323\\
3.90532334723866	-0.608064916326581\\
3.90168373075187	-0.619353685864273\\
3.87847794832012	-0.686527929275298\\
3.84432401564926	-0.773607733479556\\
3.80106343167746	-0.869580654726294\\
3.74657371571341	-0.974412774741307\\
3.67842168356242	-1.08761236299864\\
3.59392847033223	-1.20805299199347\\
3.49031115707938	-1.33378133705819\\
3.3649259564279	-1.46184032420418\\
3.21562736224065	-1.58815934848234\\
3.04123097633317	-1.70758251745685\\
2.8420227861645	-1.81411069882182\\
2.62020285358139	-1.90140695812588\\
2.38011066380083	-1.96354676116075\\
2.12808749583035	-1.99589418392156\\
1.87191250416965	-1.99589418392156\\
1.61988933619917	-1.96354676116075\\
1.37979714641861	-1.90140695812588\\
1.1579772138355	-1.81411069882182\\
1.1364146136666	-1.8039457532064\\
0.958769023666834	-1.70758251745685\\
0.784372637759352	-1.58815934848234\\
0.635074043572103	-1.46184032420418\\
0.509688842920626	-1.33378133705819\\
0.406071529667772	-1.20805299199347\\
0.321578316437581	-1.08761236299865\\
0.25342628428659	-0.974412774741307\\
0.19893656832254	-0.869580654726294\\
0.155675984350737	-0.773607733479556\\
0.121522051679879	-0.686527929275298\\
0.0946766527613359	-0.608064916326582\\
0.0736490491319272	-0.537747165580323\\
0.0572232274265433	-0.474993065158917\\
0.0444193838553086	-0.419171150914694\\
0.0344553971348155	-0.369640925964028\\
0.0267114007416175	-0.32577922591364\\
0.0206988155201003	-0.286996204010552\\
0.0160341675952417	-0.252744091623315\\
0.0158112634789833	-0.250988162794844\\
0.0124174560227557	-0.22252107962381\\
0.00961457122889731	-0.195872011619\\
0.00744317640593932	-0.17238707825341\\
0.00576147038270832	-0.151699330881393\\
0.00445930344850352	-0.133481565793814\\
0.00345118881507674	-0.117442941703917\\
0.00267082629172262	-0.103325562441295\\
0.00206682350195921	-0.0909011674757183\\
0.00159936025589764	-0.0799680127948821\\
0.000158733805841562	-0.0251974210375809\\
1.59999360002568e-05	-0.007999968000128\\
1.58740042200835e-06	-0.00251984109979014\\
1.58740104567103e-08	-0.000251984208978975\\
1.59999999935779e-09	-7.9999999968e-05\\
1.58740105193525e-10	-2.51984209968975e-05\\
1.58740105189189e-12	-2.51984209978875e-06\\
1.58740105607856e-14	-2.51984209978974e-07\\
1.6000000129382e-15	-7.99999999999999e-08\\
1.58740107556036e-16	-2.51984209978975e-08\\
1.5874012144266e-18	-2.51984209978975e-09\\
1.58740260308906e-20	-2.51984209978975e-10\\
1.58742208488702e-22	-2.51984209978975e-11\\
1.60001359111977e-23	-8.00000000000001e-12\\
1.58750499939795e-24	-2.51984209978975e-12\\
1.59057221388135e-26	-2.51984209978975e-13\\
1.58207814466889e-28	-2.51984209978975e-14\\
1.83284785865704e-30	-2.51984209978975e-15\\
1.54295828093213e-32	-2.51984209978975e-16\\
4.8985871965894e-33	-7.99999999999998e-17\\
1.54295828093213e-33	-2.51984209978975e-17\\
1.54295828093213e-34	-2.51984209978975e-18\\
1.54295828093213e-35	-2.51984209978975e-19\\
1.54295828093213e-36	-2.51984209978975e-20\\
};

\addplot [color=black, forget plot]
  table[row sep=crcr]{%
4	-0\\
4	-0\\
4	-0\\
4	-0\\
4	-0\\
4	-0\\
4	-0\\
4	6.21724893790088e-15\\
4	6.30606677987089e-14\\
4	6.3504757008559e-13\\
4	6.3495875224362e-12\\
4	7.99982302623903e-12\\
4	6.3495875224362e-11\\
4	6.34960528600459e-10\\
4	6.34960439782617e-09\\
4	6.34960422019049e-08\\
4	7.99999995138023e-08\\
3.9999999999999	6.34960420242681e-07\\
3.99999999998992	6.34960420774532e-06\\
3.99999999899206	6.34960420624523e-05\\
3.9999999984	7.99999999676445e-05\\
3.99999989920632	0.000634960404786968\\
3.999989920657	0.00634958820791304\\
3.999984000064	0.00799996800012817\\
3.99899231708027	0.0634800461094593\\
3.9984006397441	0.079968012794882\\
3.99793317649804	0.0909011674757181\\
3.99732917370828	0.103325562441295\\
3.99654881118492	0.117442941703917\\
3.9955406965515	0.133481565793814\\
3.99423852961729	0.151699330881392\\
3.99255682359406	0.17238707825341\\
3.9903854287711	0.195872011619\\
3.98758254397725	0.22252107962381\\
3.98396583240476	0.252744091623315\\
3.9793011844799	0.286996204010551\\
3.97328859925838	0.325779225913639\\
3.96554460286519	0.369640925964027\\
3.95558061614469	0.419171150914693\\
3.94277677257346	0.474993065158917\\
3.92635095086807	0.537747165580323\\
3.90532334723866	0.608064916326581\\
3.90168373075187	0.619353685864273\\
3.87847794832012	0.686527929275298\\
3.84432401564926	0.773607733479556\\
3.80106343167746	0.869580654726294\\
3.74657371571341	0.974412774741307\\
3.67842168356242	1.08761236299864\\
3.59392847033223	1.20805299199347\\
3.49031115707938	1.33378133705819\\
3.3649259564279	1.46184032420418\\
3.21562736224065	1.58815934848234\\
3.04123097633317	1.70758251745685\\
2.8420227861645	1.81411069882182\\
2.62020285358139	1.90140695812588\\
2.38011066380083	1.96354676116075\\
2.12808749583035	1.99589418392156\\
1.87191250416965	1.99589418392156\\
1.61988933619917	1.96354676116075\\
1.37979714641861	1.90140695812588\\
1.1579772138355	1.81411069882182\\
1.1364146136666	1.8039457532064\\
0.958769023666834	1.70758251745685\\
0.784372637759352	1.58815934848234\\
0.635074043572103	1.46184032420418\\
0.509688842920626	1.33378133705819\\
0.406071529667772	1.20805299199347\\
0.321578316437581	1.08761236299865\\
0.25342628428659	0.974412774741307\\
0.19893656832254	0.869580654726294\\
0.155675984350737	0.773607733479556\\
0.121522051679879	0.686527929275298\\
0.0946766527613359	0.608064916326582\\
0.0736490491319272	0.537747165580323\\
0.0572232274265433	0.474993065158917\\
0.0444193838553086	0.419171150914694\\
0.0344553971348155	0.369640925964028\\
0.0267114007416175	0.32577922591364\\
0.0206988155201003	0.286996204010552\\
0.0160341675952417	0.252744091623315\\
0.0158112634789833	0.250988162794844\\
0.0124174560227557	0.22252107962381\\
0.00961457122889731	0.195872011619\\
0.00744317640593932	0.17238707825341\\
0.00576147038270832	0.151699330881393\\
0.00445930344850352	0.133481565793814\\
0.00345118881507674	0.117442941703917\\
0.00267082629172262	0.103325562441295\\
0.00206682350195921	0.0909011674757183\\
0.00159936025589764	0.0799680127948821\\
0.000158733805841562	0.0251974210375809\\
1.59999360002568e-05	0.007999968000128\\
1.58740042200835e-06	0.00251984109979014\\
1.58740104567103e-08	0.000251984208978975\\
1.59999999935779e-09	7.9999999968e-05\\
1.58740105193525e-10	2.51984209968975e-05\\
1.58740105189189e-12	2.51984209978875e-06\\
1.58740105607856e-14	2.51984209978974e-07\\
1.6000000129382e-15	7.99999999999999e-08\\
1.58740107556036e-16	2.51984209978975e-08\\
1.5874012144266e-18	2.51984209978975e-09\\
1.58740260308906e-20	2.51984209978975e-10\\
1.58742208488702e-22	2.51984209978975e-11\\
1.60001359111977e-23	8.00000000000001e-12\\
1.58750499939795e-24	2.51984209978975e-12\\
1.59057221388135e-26	2.51984209978975e-13\\
1.58207814466889e-28	2.51984209978975e-14\\
1.83284785865704e-30	2.51984209978975e-15\\
1.54295828093213e-32	2.51984209978975e-16\\
4.8985871965894e-33	7.99999999999998e-17\\
1.54295828093213e-33	2.51984209978975e-17\\
1.54295828093213e-34	2.51984209978975e-18\\
1.54295828093213e-35	2.51984209978975e-19\\
1.54295828093213e-36	2.51984209978975e-20\\
};

\end{axis}
\end{tikzpicture}%
\end{minipage} \hspace{.2cm} \hfill %
\begin{minipage}{.6\columnwidth}
    \setlength{\figurewidth}{.9\columnwidth}
    \setlength{\figureheight}{0.7\columnwidth}
    \centering
%
%
\begin{tikzpicture}

\begin{axis}[%
width=0.951\figurewidth,
height=\figureheight,
at={(0\figurewidth,0\figureheight)},
scale only axis,
xmin=-3.4333,
xmax=4.1,
xlabel style={font=\color{white!15!black}},
xlabel={real part},
xlabel near ticks,
ymin=-2.8919,
ymax=2.8919,
legend style={legend cell align=left, align=left, draw=white!15!black},
axis equal
]
\begin{pgfonlayer}{background}
  \draw[fill=black!20!white, draw=black!20!white] (axis cs:-3.43333333333333,-3.05) rectangle (axis cs:-1.33333333333333,3.06);
\end{pgfonlayer}
\addplot [color=black,forget plot]
  table[row sep=crcr]{%
4	0\\
4	0\\
4	0\\
4	0\\
4	0\\
4	0\\
4	0\\
4	-6.21724893790088e-15\\
4	-6.30606677987089e-14\\
4	-6.3504757008559e-13\\
4	-6.3495875224362e-12\\
4	-7.99982302623903e-12\\
4	-6.3495875224362e-11\\
4	-6.34960528600459e-10\\
4	-6.34960439782617e-09\\
4	-6.34960422019049e-08\\
4	-7.99999995138023e-08\\
3.9999999999999	-6.34960420242681e-07\\
3.99999999998992	-6.34960420774532e-06\\
3.99999999899206	-6.34960420624523e-05\\
3.9999999984	-7.99999999676445e-05\\
3.99999989920632	-0.000634960404786968\\
3.999989920657	-0.00634958820791304\\
3.999984000064	-0.00799996800012817\\
3.99899231708027	-0.0634800461094593\\
3.9984006397441	-0.079968012794882\\
3.99793317649804	-0.0909011674757181\\
3.99732917370828	-0.103325562441295\\
3.99654881118492	-0.117442941703917\\
3.9955406965515	-0.133481565793814\\
3.99423852961729	-0.151699330881392\\
3.99255682359406	-0.17238707825341\\
3.9903854287711	-0.195872011619\\
3.98758254397725	-0.22252107962381\\
3.98396583240476	-0.252744091623315\\
3.9793011844799	-0.286996204010551\\
3.97328859925838	-0.325779225913639\\
3.96554460286519	-0.369640925964027\\
3.95558061614469	-0.419171150914693\\
3.94277677257346	-0.474993065158917\\
3.92635095086807	-0.537747165580323\\
3.90532334723866	-0.608064916326581\\
3.90168373075187	-0.619353685864273\\
3.87847794832012	-0.686527929275298\\
3.84432401564926	-0.773607733479556\\
3.80106343167746	-0.869580654726294\\
3.74657371571341	-0.974412774741307\\
3.67842168356242	-1.08761236299864\\
3.59392847033223	-1.20805299199347\\
3.49031115707938	-1.33378133705819\\
3.3649259564279	-1.46184032420418\\
3.21562736224065	-1.58815934848234\\
3.04123097633317	-1.70758251745685\\
2.8420227861645	-1.81411069882182\\
2.62020285358139	-1.90140695812588\\
2.38011066380083	-1.96354676116075\\
2.12808749583035	-1.99589418392156\\
1.87191250416965	-1.99589418392156\\
1.61988933619917	-1.96354676116075\\
1.37979714641861	-1.90140695812588\\
1.1579772138355	-1.81411069882182\\
1.1364146136666	-1.8039457532064\\
0.958769023666834	-1.70758251745685\\
0.784372637759352	-1.58815934848234\\
0.635074043572103	-1.46184032420418\\
0.509688842920626	-1.33378133705819\\
0.406071529667772	-1.20805299199347\\
0.321578316437581	-1.08761236299865\\
0.25342628428659	-0.974412774741307\\
0.19893656832254	-0.869580654726294\\
0.155675984350737	-0.773607733479556\\
0.121522051679879	-0.686527929275298\\
0.0946766527613359	-0.608064916326582\\
0.0736490491319272	-0.537747165580323\\
0.0572232274265433	-0.474993065158917\\
0.0444193838553086	-0.419171150914694\\
0.0344553971348155	-0.369640925964028\\
0.0267114007416175	-0.32577922591364\\
0.0206988155201003	-0.286996204010552\\
0.0160341675952417	-0.252744091623315\\
0.0158112634789833	-0.250988162794844\\
0.0124174560227557	-0.22252107962381\\
0.00961457122889731	-0.195872011619\\
0.00744317640593932	-0.17238707825341\\
0.00576147038270832	-0.151699330881393\\
0.00445930344850352	-0.133481565793814\\
0.00345118881507674	-0.117442941703917\\
0.00267082629172262	-0.103325562441295\\
0.00206682350195921	-0.0909011674757183\\
0.00159936025589764	-0.0799680127948821\\
0.000158733805841562	-0.0251974210375809\\
1.59999360002568e-05	-0.007999968000128\\
1.58740042200835e-06	-0.00251984109979014\\
1.58740104567103e-08	-0.000251984208978975\\
1.59999999935779e-09	-7.9999999968e-05\\
1.58740105193525e-10	-2.51984209968975e-05\\
1.58740105189189e-12	-2.51984209978875e-06\\
1.58740105607856e-14	-2.51984209978974e-07\\
1.6000000129382e-15	-7.99999999999999e-08\\
1.58740107556036e-16	-2.51984209978975e-08\\
1.5874012144266e-18	-2.51984209978975e-09\\
1.58740260308906e-20	-2.51984209978975e-10\\
1.58742208488702e-22	-2.51984209978975e-11\\
1.60001359111977e-23	-8.00000000000001e-12\\
1.58750499939795e-24	-2.51984209978975e-12\\
1.59057221388135e-26	-2.51984209978975e-13\\
1.58207814466889e-28	-2.51984209978975e-14\\
1.83284785865704e-30	-2.51984209978975e-15\\
1.54295828093213e-32	-2.51984209978975e-16\\
4.8985871965894e-33	-7.99999999999998e-17\\
1.54295828093213e-33	-2.51984209978975e-17\\
1.54295828093213e-34	-2.51984209978975e-18\\
1.54295828093213e-35	-2.51984209978975e-19\\
1.54295828093213e-36	-2.51984209978975e-20\\
};

\addplot [color=black,forget plot]
  table[row sep=crcr]{%
4	-0\\
4	-0\\
4	-0\\
4	-0\\
4	-0\\
4	-0\\
4	-0\\
4	6.21724893790088e-15\\
4	6.30606677987089e-14\\
4	6.3504757008559e-13\\
4	6.3495875224362e-12\\
4	7.99982302623903e-12\\
4	6.3495875224362e-11\\
4	6.34960528600459e-10\\
4	6.34960439782617e-09\\
4	6.34960422019049e-08\\
4	7.99999995138023e-08\\
3.9999999999999	6.34960420242681e-07\\
3.99999999998992	6.34960420774532e-06\\
3.99999999899206	6.34960420624523e-05\\
3.9999999984	7.99999999676445e-05\\
3.99999989920632	0.000634960404786968\\
3.999989920657	0.00634958820791304\\
3.999984000064	0.00799996800012817\\
3.99899231708027	0.0634800461094593\\
3.9984006397441	0.079968012794882\\
3.99793317649804	0.0909011674757181\\
3.99732917370828	0.103325562441295\\
3.99654881118492	0.117442941703917\\
3.9955406965515	0.133481565793814\\
3.99423852961729	0.151699330881392\\
3.99255682359406	0.17238707825341\\
3.9903854287711	0.195872011619\\
3.98758254397725	0.22252107962381\\
3.98396583240476	0.252744091623315\\
3.9793011844799	0.286996204010551\\
3.97328859925838	0.325779225913639\\
3.96554460286519	0.369640925964027\\
3.95558061614469	0.419171150914693\\
3.94277677257346	0.474993065158917\\
3.92635095086807	0.537747165580323\\
3.90532334723866	0.608064916326581\\
3.90168373075187	0.619353685864273\\
3.87847794832012	0.686527929275298\\
3.84432401564926	0.773607733479556\\
3.80106343167746	0.869580654726294\\
3.74657371571341	0.974412774741307\\
3.67842168356242	1.08761236299864\\
3.59392847033223	1.20805299199347\\
3.49031115707938	1.33378133705819\\
3.3649259564279	1.46184032420418\\
3.21562736224065	1.58815934848234\\
3.04123097633317	1.70758251745685\\
2.8420227861645	1.81411069882182\\
2.62020285358139	1.90140695812588\\
2.38011066380083	1.96354676116075\\
2.12808749583035	1.99589418392156\\
1.87191250416965	1.99589418392156\\
1.61988933619917	1.96354676116075\\
1.37979714641861	1.90140695812588\\
1.1579772138355	1.81411069882182\\
1.1364146136666	1.8039457532064\\
0.958769023666834	1.70758251745685\\
0.784372637759352	1.58815934848234\\
0.635074043572103	1.46184032420418\\
0.509688842920626	1.33378133705819\\
0.406071529667772	1.20805299199347\\
0.321578316437581	1.08761236299865\\
0.25342628428659	0.974412774741307\\
0.19893656832254	0.869580654726294\\
0.155675984350737	0.773607733479556\\
0.121522051679879	0.686527929275298\\
0.0946766527613359	0.608064916326582\\
0.0736490491319272	0.537747165580323\\
0.0572232274265433	0.474993065158917\\
0.0444193838553086	0.419171150914694\\
0.0344553971348155	0.369640925964028\\
0.0267114007416175	0.32577922591364\\
0.0206988155201003	0.286996204010552\\
0.0160341675952417	0.252744091623315\\
0.0158112634789833	0.250988162794844\\
0.0124174560227557	0.22252107962381\\
0.00961457122889731	0.195872011619\\
0.00744317640593932	0.17238707825341\\
0.00576147038270832	0.151699330881393\\
0.00445930344850352	0.133481565793814\\
0.00345118881507674	0.117442941703917\\
0.00267082629172262	0.103325562441295\\
0.00206682350195921	0.0909011674757183\\
0.00159936025589764	0.0799680127948821\\
0.000158733805841562	0.0251974210375809\\
1.59999360002568e-05	0.007999968000128\\
1.58740042200835e-06	0.00251984109979014\\
1.58740104567103e-08	0.000251984208978975\\
1.59999999935779e-09	7.9999999968e-05\\
1.58740105193525e-10	2.51984209968975e-05\\
1.58740105189189e-12	2.51984209978875e-06\\
1.58740105607856e-14	2.51984209978974e-07\\
1.6000000129382e-15	7.99999999999999e-08\\
1.58740107556036e-16	2.51984209978975e-08\\
1.5874012144266e-18	2.51984209978975e-09\\
1.58740260308906e-20	2.51984209978975e-10\\
1.58742208488702e-22	2.51984209978975e-11\\
1.60001359111977e-23	8.00000000000001e-12\\
1.58750499939795e-24	2.51984209978975e-12\\
1.59057221388135e-26	2.51984209978975e-13\\
1.58207814466889e-28	2.51984209978975e-14\\
1.83284785865704e-30	2.51984209978975e-15\\
1.54295828093213e-32	2.51984209978975e-16\\
4.8985871965894e-33	7.99999999999998e-17\\
1.54295828093213e-33	2.51984209978975e-17\\
1.54295828093213e-34	2.51984209978975e-18\\
1.54295828093213e-35	2.51984209978975e-19\\
1.54295828093213e-36	2.51984209978975e-20\\
};

\end{axis}
\end{tikzpicture}%
\end{minipage}\hfill %
\begin{minipage}{.6\columnwidth}
     \setlength{\figurewidth}{.9\columnwidth}
    \setlength{\figureheight}{0.7\columnwidth}
    \centering
%
%
\begin{tikzpicture}

\begin{axis}[%
width=0.951\figurewidth,
height=\figureheight,
at={(0\figurewidth,0\figureheight)},
scale only axis,
xmin=-2.0700408997955,
xmax=4.90337423312883,
xlabel style={font=\color{white!15!black}},
xlabel={real part},
xlabel near ticks,
ymin=-2.75,
ymax=2.75,
legend style={legend cell align=left, align=left, draw=white!15!black},
disabledatascaling,
axis equal
]
\begin{pgfonlayer}{background}
  \draw[fill=black!20!white, draw=black!20!white] (axis cs:1.4167,0) circle (2.75);
\end{pgfonlayer}
\addplot [color=black,forget plot]
  table[row sep=crcr]{%
4	0\\
4	0\\
4	0\\
4	0\\
4	0\\
4	0\\
4	0\\
4	-6.21724893790088e-15\\
4	-6.30606677987089e-14\\
4	-6.3504757008559e-13\\
4	-6.3495875224362e-12\\
4	-7.99982302623903e-12\\
4	-6.3495875224362e-11\\
4	-6.34960528600459e-10\\
4	-6.34960439782617e-09\\
4	-6.34960422019049e-08\\
4	-7.99999995138023e-08\\
3.9999999999999	-6.34960420242681e-07\\
3.99999999998992	-6.34960420774532e-06\\
3.99999999899206	-6.34960420624523e-05\\
3.9999999984	-7.99999999676445e-05\\
3.99999989920632	-0.000634960404786968\\
3.999989920657	-0.00634958820791304\\
3.999984000064	-0.00799996800012817\\
3.99899231708027	-0.0634800461094593\\
3.9984006397441	-0.079968012794882\\
3.99793317649804	-0.0909011674757181\\
3.99732917370828	-0.103325562441295\\
3.99654881118492	-0.117442941703917\\
3.9955406965515	-0.133481565793814\\
3.99423852961729	-0.151699330881392\\
3.99255682359406	-0.17238707825341\\
3.9903854287711	-0.195872011619\\
3.98758254397725	-0.22252107962381\\
3.98396583240476	-0.252744091623315\\
3.9793011844799	-0.286996204010551\\
3.97328859925838	-0.325779225913639\\
3.96554460286519	-0.369640925964027\\
3.95558061614469	-0.419171150914693\\
3.94277677257346	-0.474993065158917\\
3.92635095086807	-0.537747165580323\\
3.90532334723866	-0.608064916326581\\
3.90168373075187	-0.619353685864273\\
3.87847794832012	-0.686527929275298\\
3.84432401564926	-0.773607733479556\\
3.80106343167746	-0.869580654726294\\
3.74657371571341	-0.974412774741307\\
3.67842168356242	-1.08761236299864\\
3.59392847033223	-1.20805299199347\\
3.49031115707938	-1.33378133705819\\
3.3649259564279	-1.46184032420418\\
3.21562736224065	-1.58815934848234\\
3.04123097633317	-1.70758251745685\\
2.8420227861645	-1.81411069882182\\
2.62020285358139	-1.90140695812588\\
2.38011066380083	-1.96354676116075\\
2.12808749583035	-1.99589418392156\\
1.87191250416965	-1.99589418392156\\
1.61988933619917	-1.96354676116075\\
1.37979714641861	-1.90140695812588\\
1.1579772138355	-1.81411069882182\\
1.1364146136666	-1.8039457532064\\
0.958769023666834	-1.70758251745685\\
0.784372637759352	-1.58815934848234\\
0.635074043572103	-1.46184032420418\\
0.509688842920626	-1.33378133705819\\
0.406071529667772	-1.20805299199347\\
0.321578316437581	-1.08761236299865\\
0.25342628428659	-0.974412774741307\\
0.19893656832254	-0.869580654726294\\
0.155675984350737	-0.773607733479556\\
0.121522051679879	-0.686527929275298\\
0.0946766527613359	-0.608064916326582\\
0.0736490491319272	-0.537747165580323\\
0.0572232274265433	-0.474993065158917\\
0.0444193838553086	-0.419171150914694\\
0.0344553971348155	-0.369640925964028\\
0.0267114007416175	-0.32577922591364\\
0.0206988155201003	-0.286996204010552\\
0.0160341675952417	-0.252744091623315\\
0.0158112634789833	-0.250988162794844\\
0.0124174560227557	-0.22252107962381\\
0.00961457122889731	-0.195872011619\\
0.00744317640593932	-0.17238707825341\\
0.00576147038270832	-0.151699330881393\\
0.00445930344850352	-0.133481565793814\\
0.00345118881507674	-0.117442941703917\\
0.00267082629172262	-0.103325562441295\\
0.00206682350195921	-0.0909011674757183\\
0.00159936025589764	-0.0799680127948821\\
0.000158733805841562	-0.0251974210375809\\
1.59999360002568e-05	-0.007999968000128\\
1.58740042200835e-06	-0.00251984109979014\\
1.58740104567103e-08	-0.000251984208978975\\
1.59999999935779e-09	-7.9999999968e-05\\
1.58740105193525e-10	-2.51984209968975e-05\\
1.58740105189189e-12	-2.51984209978875e-06\\
1.58740105607856e-14	-2.51984209978974e-07\\
1.6000000129382e-15	-7.99999999999999e-08\\
1.58740107556036e-16	-2.51984209978975e-08\\
1.5874012144266e-18	-2.51984209978975e-09\\
1.58740260308906e-20	-2.51984209978975e-10\\
1.58742208488702e-22	-2.51984209978975e-11\\
1.60001359111977e-23	-8.00000000000001e-12\\
1.58750499939795e-24	-2.51984209978975e-12\\
1.59057221388135e-26	-2.51984209978975e-13\\
1.58207814466889e-28	-2.51984209978975e-14\\
1.83284785865704e-30	-2.51984209978975e-15\\
1.54295828093213e-32	-2.51984209978975e-16\\
4.8985871965894e-33	-7.99999999999998e-17\\
1.54295828093213e-33	-2.51984209978975e-17\\
1.54295828093213e-34	-2.51984209978975e-18\\
1.54295828093213e-35	-2.51984209978975e-19\\
1.54295828093213e-36	-2.51984209978975e-20\\
};

\addplot [color=black,forget plot]
  table[row sep=crcr]{%
4	-0\\
4	-0\\
4	-0\\
4	-0\\
4	-0\\
4	-0\\
4	-0\\
4	6.21724893790088e-15\\
4	6.30606677987089e-14\\
4	6.3504757008559e-13\\
4	6.3495875224362e-12\\
4	7.99982302623903e-12\\
4	6.3495875224362e-11\\
4	6.34960528600459e-10\\
4	6.34960439782617e-09\\
4	6.34960422019049e-08\\
4	7.99999995138023e-08\\
3.9999999999999	6.34960420242681e-07\\
3.99999999998992	6.34960420774532e-06\\
3.99999999899206	6.34960420624523e-05\\
3.9999999984	7.99999999676445e-05\\
3.99999989920632	0.000634960404786968\\
3.999989920657	0.00634958820791304\\
3.999984000064	0.00799996800012817\\
3.99899231708027	0.0634800461094593\\
3.9984006397441	0.079968012794882\\
3.99793317649804	0.0909011674757181\\
3.99732917370828	0.103325562441295\\
3.99654881118492	0.117442941703917\\
3.9955406965515	0.133481565793814\\
3.99423852961729	0.151699330881392\\
3.99255682359406	0.17238707825341\\
3.9903854287711	0.195872011619\\
3.98758254397725	0.22252107962381\\
3.98396583240476	0.252744091623315\\
3.9793011844799	0.286996204010551\\
3.97328859925838	0.325779225913639\\
3.96554460286519	0.369640925964027\\
3.95558061614469	0.419171150914693\\
3.94277677257346	0.474993065158917\\
3.92635095086807	0.537747165580323\\
3.90532334723866	0.608064916326581\\
3.90168373075187	0.619353685864273\\
3.87847794832012	0.686527929275298\\
3.84432401564926	0.773607733479556\\
3.80106343167746	0.869580654726294\\
3.74657371571341	0.974412774741307\\
3.67842168356242	1.08761236299864\\
3.59392847033223	1.20805299199347\\
3.49031115707938	1.33378133705819\\
3.3649259564279	1.46184032420418\\
3.21562736224065	1.58815934848234\\
3.04123097633317	1.70758251745685\\
2.8420227861645	1.81411069882182\\
2.62020285358139	1.90140695812588\\
2.38011066380083	1.96354676116075\\
2.12808749583035	1.99589418392156\\
1.87191250416965	1.99589418392156\\
1.61988933619917	1.96354676116075\\
1.37979714641861	1.90140695812588\\
1.1579772138355	1.81411069882182\\
1.1364146136666	1.8039457532064\\
0.958769023666834	1.70758251745685\\
0.784372637759352	1.58815934848234\\
0.635074043572103	1.46184032420418\\
0.509688842920626	1.33378133705819\\
0.406071529667772	1.20805299199347\\
0.321578316437581	1.08761236299865\\
0.25342628428659	0.974412774741307\\
0.19893656832254	0.869580654726294\\
0.155675984350737	0.773607733479556\\
0.121522051679879	0.686527929275298\\
0.0946766527613359	0.608064916326582\\
0.0736490491319272	0.537747165580323\\
0.0572232274265433	0.474993065158917\\
0.0444193838553086	0.419171150914694\\
0.0344553971348155	0.369640925964028\\
0.0267114007416175	0.32577922591364\\
0.0206988155201003	0.286996204010552\\
0.0160341675952417	0.252744091623315\\
0.0158112634789833	0.250988162794844\\
0.0124174560227557	0.22252107962381\\
0.00961457122889731	0.195872011619\\
0.00744317640593932	0.17238707825341\\
0.00576147038270832	0.151699330881393\\
0.00445930344850352	0.133481565793814\\
0.00345118881507674	0.117442941703917\\
0.00267082629172262	0.103325562441295\\
0.00206682350195921	0.0909011674757183\\
0.00159936025589764	0.0799680127948821\\
0.000158733805841562	0.0251974210375809\\
1.59999360002568e-05	0.007999968000128\\
1.58740042200835e-06	0.00251984109979014\\
1.58740104567103e-08	0.000251984208978975\\
1.59999999935779e-09	7.9999999968e-05\\
1.58740105193525e-10	2.51984209968975e-05\\
1.58740105189189e-12	2.51984209978875e-06\\
1.58740105607856e-14	2.51984209978974e-07\\
1.6000000129382e-15	7.99999999999999e-08\\
1.58740107556036e-16	2.51984209978975e-08\\
1.5874012144266e-18	2.51984209978975e-09\\
1.58740260308906e-20	2.51984209978975e-10\\
1.58742208488702e-22	2.51984209978975e-11\\
1.60001359111977e-23	8.00000000000001e-12\\
1.58750499939795e-24	2.51984209978975e-12\\
1.59057221388135e-26	2.51984209978975e-13\\
1.58207814466889e-28	2.51984209978975e-14\\
1.83284785865704e-30	2.51984209978975e-15\\
1.54295828093213e-32	2.51984209978975e-16\\
4.8985871965894e-33	7.99999999999998e-17\\
1.54295828093213e-33	2.51984209978975e-17\\
1.54295828093213e-34	2.51984209978975e-18\\
1.54295828093213e-35	2.51984209978975e-19\\
1.54295828093213e-36	2.51984209978975e-20\\
};

\end{axis}
\end{tikzpicture}%
\end{minipage}\hfill %
\mbox{} 
    \caption{The figure illustrates that the conditions for \eqref{eq:PRfunction} to be strictly positive real are indeed met as long as $\alpha_\text{s}>0$. The different cases $\alpha_\text{s}>1/\kappa$, $\alpha_\text{s}=1/\kappa$, and $\alpha_\text{s}<1/\kappa$ are shown on the left, center, and right, respectively. The graph of $P_\text{T}(i\omega)$, $\omega \in \mathbb{R}$ is drawn in black (solid), whereas the disk $D$ is shown in grey. The condition number $\kappa$ is set to $4$, and $\alpha_\text{s}$ is set to $0.55$ (left), $0.25$ (center) and $0.01$ (right). }
    \label{Fig:dC}
\end{figure*}
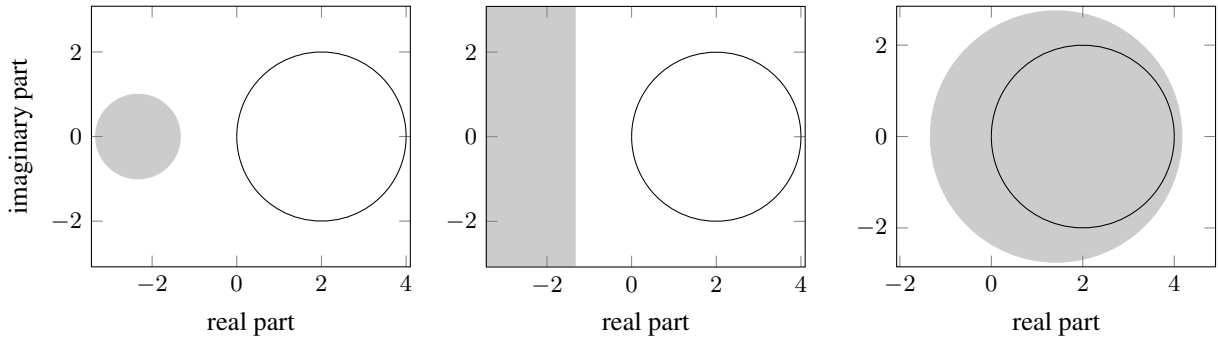

\section{Proof of Prop.~\ref{Prop:Conv}}\label{App:Prop3p3}
We rewrite the dynamics \eqref{eq:Lure} as a Lur\'{e} problem with $\nabla f(x)/L - x/\kappa$ as the nonlinear feedback term, as illustrated in Fig.~\ref{Fig:Lure}. The linear system $P_\text{T}$ is given by \eqref{eq:feedback}, where we slightly abuse notation and replace the single components $x_j(t)\in \mathbb{R}, \dot{x}_j(t)\in \mathbb{R}, \dots$ by the vectors $x(t)\in \mathbb{R}^n, \dot{x}(t) \in \mathbb{R}^n, \dots$. By assumption, the nonlinear feedback term is sector bounded in $[\alpha_\text{s}-1/\kappa,1-1/\kappa]$ \citep[p.~232, Def.~6.2]{Khalil}. The circle criterion \citep[p.~265]{Khalil} implies that the dynamics \eqref{eq:Lure} are globally asymptotically stable in the sense of Lyapunov if the transfer function
\begin{equation}
(1 + (1-1/\kappa) P_\text{T}(s))~(1 + (\alpha_\text{s}-1/\kappa)  P_\text{T}(s))^{-1} \label{eq:PRfunction}
\end{equation}
is strictly positive real. This has the following graphical interpretation \citep[p.~268]{Khalil}: Let $D\subset \mathbb{C}$ be the closed disk connecting the two points $-1/(\alpha_\text{s} - 1/\kappa)$ and $-1/(1-1/\kappa)$ on the real line. Then, the transfer function \eqref{eq:PRfunction} is strictly positive real if
\begin{itemize}
    \item case $\alpha_\text{s} > 1/\kappa$: the graph of $P_\text{T}(i\omega)$, $\omega \in \mathbb{R}$ does not enter $D$;
    \item case $\alpha_\text{s} = 1/\kappa$: the graph of $P_\text{T}(i\omega)$, $\omega \in \mathbb{R}$ lies strictly to the right of the vertical line that crosses the point $-1/(1-1/\kappa)$ on the real line;
    \item case $\alpha_\text{s} < 1/\kappa$: the graph of $P_\text{T}(i\omega)$ is contained in the interior of $D$.
\end{itemize}
The three different cases are illustrated in Fig.~\ref{Fig:dC}.
It remains to evaluate $P_\text{T}(i\omega)$. Due to the fact that $g_j$ and $h_j$ are chosen according to \eqref{eq:fastAlg}, we obtain
\begin{equation}
    P_\text{T}(s)=\frac{\kappa^{1-1/k}}{s+1/\kappa^{1/k}},
\end{equation}
where the common factor $(s+1/\kappa^{1/k})^{k-1}$ of the numerator and denominator cancels out.\footnote{The proof of the circle criterion relies on the Kalman-Yakubovich-Popov Lemma, which does not apply when there is a cancellation in the numerator and denominator of $P_\text{T}(s)$ (i.e. in case the realization is non-minimal). However, the lemma can be generalized to accommodate non-minimal realizations that are stabilizable or detectable as done in \citet{PRL}.} Thus, the graph of the function $P_\text{T}(i\omega)$ for $\omega\in \mathbb{R}$ is given by the circle connecting the origin with the point $\kappa$ on the real line. The situation is illustrated in Fig.~\ref{Fig:dC}, and it can be verified that \eqref{eq:PRfunction} is indeed strictly positive real as long as $\alpha_\text{s}>0$. This concludes the first part of the proof.

In view of Fig.~\ref{Fig:dC}, we realize that for $\alpha_\text{s} \geq 1/\kappa$ the graph of $P_\text{T}(j\omega)$, $\omega\in \mathbb{R}$, is nowhere close to the stability boundary indicated by the shaded regions. In fact, we can replace $P_\text{T}(s)$ by $P_\text{T}(s-1/\kappa^{1/k})$ in \eqref{eq:PRfunction}, which, for $\alpha_\text{s}=1/\kappa$, reduces to
\begin{equation*}
1+(1-1/\kappa) P_\text{T}(s-1/\kappa^{1/k}),
\end{equation*}
and can be verified to be positive real. This implies that the dynamics \eqref{eq:Lure} converge at least with rate $1/\kappa^{1/k}$.

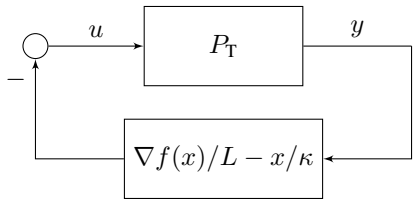
\begin{figure}
    \centering
    \tikzstyle{block} = [draw, rectangle, 
    minimum height=3em, minimum width=6em]
\tikzstyle{sum} = [draw, circle, node distance=1cm]
\tikzstyle{input} = [coordinate]
\tikzstyle{output} = [coordinate]
\tikzstyle{pinstyle} = [pin edge={to-,thin,black}]

\begin{tikzpicture}[auto, node distance=2cm,>=latex']
    \node [block] (system) {$P_\text{T}$};
            								 
    \node [block, below of=system,node distance=1.5cm] (grad) {$\nabla f(x)/L-x/\kappa$};

    \node [sum, left of=system, node distance=2.5cm] (ln) {};
 	\node (rn) [coordinate,right of=system, node distance=2.5cm] {};
 	\draw [-] (system) -- node[name=y]{$y$} (rn);   
 	\draw [->] (rn) |- (grad); 
	\draw [->] (grad) -| node [pos=0.9] {$-$} (ln);
	\draw [->] (ln) -- node[name=u]{$u$} (system);
\end{tikzpicture}
    \caption{Block diagram that illustrates the dynamics \eqref{eq:Lure}.}
    \label{Fig:Lure}
\end{figure}

\end{document}